\documentclass{scrartcl}
\usepackage{amsmath,amssymb,amsfonts,amsbsy,setspace, amsthm,bm}
\usepackage{marginnote}
\usepackage[normalem]{ulem}
\usepackage{color}
\usepackage{hyperref}
\usepackage{enumerate}
\usepackage{eucal}
\usepackage{algorithm}
\usepackage{algpseudocode}
\usepackage{graphicx}

\everymath{\displaystyle} 

\newtheorem{thm}{Theorem}[section]
\newtheorem{lem}[thm]{Lemma}

\newtheorem{prop}[thm]{Proposition}
\numberwithin{equation}{section} 
\numberwithin{figure}{section} 
 
\theoremstyle{remark}
\newtheorem{rmk}[thm]{Remark}

\newtheorem*{acknowledgement*}{Acknowledgement}

\theoremstyle{definition}

\def\eps{\varepsilon}
\newcommand{\R}{\mathbb R}
\newcommand{\Rd}{\mathbb R^{d}}
\newcommand{\C}{\mathcal{C}}

\newcommand{\N}{\mathbb N}

\newcommand{\W}{\mathcal{W}}
\newcommand{\A}{\mathcal{A}}
\newcommand{\leb}{\mathrm{ Leb}}
\newcommand{\en}{\mathcal{H}}
\newcommand{\Pb}{\mathcal{P}}
\newcommand{\T}{\mathbb T}

\newcommand{\roi}{\rho ^i}

\newcommand{\roj}{\rho ^j}
\newcommand{\ui}{u ^i}

\newcommand{\dt}{\partial_t}
\DeclareMathOperator{\Div}{div}
\newcommand{\jc}{\mathcal J}
\newcommand{\bro}{\underline \rho}
\newcommand{\vij}{V^{i,j}}
\newcommand{\vji}{V^{j,i}}
\newcommand{\fc}{\mathcal F}
\newcommand{\kc}{\mathcal K}
\newcommand{\ec}{\mathcal E}

\newcommand{\ac}{\mathcal A}
\newcommand{\robf}{\bm{\rho}}
\newcommand{\ropc}{\underline \rho}

\newcommand{\roam}{\tilde \rho}
\newcommand{\vam}{\tilde v}
\newcommand{\mam}{\tilde m}
\DeclareMathOperator{\argmin}{argmin}
\newcommand{\dd}{\mathrm{d}}
\DeclareOldFontCommand{\bf}{\normalfont\bfseries}{\mathbf}
\newcommand{\TabFour}[1]{ %
\begin{tabular}{@{}c@{\hspace{1mm}}c@{\hspace{1mm}}c@{\hspace{1mm}}c@{\hspace{1mm}}}
 #1
\end{tabular}
}

\begin{document}

\title{A variational formulation of a Multi-Population Mean Field Games with non-local interactions}

\author{Luigi De Pascale\thanks{Dipartimento di Matematica ed Informatica,
Universit\'a di Firenze, Viale Morgagni, 67/a - 50134 Firenze, Italy.},\; 
Luca Nenna\thanks{Universit\'e Paris-Saclay, CNRS, Laboratoire de math\'ematiques d'Orsay, ParMA, Inria Saclay, 91405, Orsay, France. 
email: luca.nenna@universite-paris-saclay.fr}}

\maketitle

\begin{abstract}
We propose a MFG model with quadratic Hamiltonian involving $N$ populations. This results in a system of $N$ Hamilton-Jacobi-Bellman and $N$ Fokker-Planck equations with non-local interactions. As in the classical case we introduce an Eulerian variational formulation which, despite the non convexity of the interaction, still gives a weak solution to the MFG model. The problem can be reformulated in Lagrangian terms and solved numerically by a Sinkhorn-like scheme. We present numerical results based on this approach, these simulations exhibit different behaviors depending on the nature (repulsive or attractive) of the non-local interaction.

\end{abstract}

\vskip\baselineskip\noindent
\date{\today}

\section{Introduction}

The aim of this paper is to establish a variational formulation (both Lagrangian and Eulerian), as well as a suitable numerical methods, for quadratic second-order Mean Field Games which involve $N$ different populations interacting through a given non-local functional.  
Let $d$ be the dimension of the space; then for $i=1, \dots, N$ we consider the following coupled PDE system.
\begin{equation}\label{mfg}
\boxed{\left\{ \begin{array}{l}
-\dt \ui-\frac{1}{2} \Delta \ui +\frac{1}{2} |\nabla \ui|^2 = \sum_{j \neq i}\int_{\R^d} \bigg(\vij(x-y)+\vji(x-y)\bigg)\roj \dd y, \\
\dt \roi -\frac{1}{2} \Delta \roi -\Div (\nabla \ui \roi)=0,\\
\roi(0,x)=\roi_0(x), \ \ui(T,x)= g^i(x);\\
\end{array}\right.}
\end{equation}
where $\roi_0\in\mathcal{P}_{ac}(\R^d)$ and $g^i \in \C^0 (\R^d)$ and $\vij$ is a given positive and lower semi-continuous potential, the full set of assumptions is given at the end of this introduction.  In the paper, we will mostly use the more compact notation $\rho_t$ for $\rho(t,x)$.  Since $t\mapsto \rho_t$ will be an absolutely continuous curve with values in the space of probability measure equipped with a Wasserstein distance (see section 2) the initial datum $\rho_0$ will be assumed in a straightforward sense.\\
Mean-field games involving several populations have been studied recently by \cite{cirant2015multi,cirant2017bifurcation,cardaliaguet2021introduction,barilla2021mean,bensoussan2018mean,dayanikli2023multi,sun2022mean,meszaros2018variational}; for some more details about the theory of MFG we refer the reader to the seminal work by Lasry and Lions \cite{lasry2007mean}, the book \cite{carmona2018probabilistic} or the lecture notes by Cardaliaguet \cite{cardaliaguet2018short}. 
Notice that the interaction of the populations can be expressed via a more general functional, but we will not discuss here the extra difficulties, or via some Optimal Transport coupling as done in \cite{barilla2021mean}.
We will show that the above system can be seen as optimality conditions for the following \emph{Eulerian} variational problem .
\begin{equation*}
\boxed{\inf \left\{  \jc(\rho^1, v^1; \dots; \rho^N, v^N)    \ | \ \
\partial_t \roi_t - \frac12 \Delta \roi_t + \Div (\rho_t^iv_t^i) =0,\ \roi(0,x)=\roi_0,\forall i\right\}}
\end{equation*}
where

\begin{multline}
\label{eq:eulerian_1}
	\jc(\rho^1, v^1; \dots; \rho^N, v^N):= \sum_i \int _{[0,T]\times \Rd } \frac{|v^i|^2}{2} \dd \roi_t \dd t +\\ \sum_{\stackrel{1\leq j \leq N}{j \neq i }}   \int_0^T \int_{\R^d \times \R^d} \frac 12\bigg(\vij(x-y)+\vji(x-y)\bigg) \dd \roi_t \otimes \dd \roj_t \dd t + \int_{\R^d} g^i(x) \dd \rho^i (T,x).
\end{multline}

Moreover, we will also relate the above minimization problem with a \emph{Lagrangian} relative entropy minimization problem, that is

\begin{equation}
\boxed{\min\{ J(Q^1, \dots, Q^N) \ : \ e^0_\sharp Q^i= \rho_0^i \}}
\end{equation}
where 
\begin{eqnarray*}
J(Q^1, \dots, Q^N) &:= & \sum_i \en(Q^i | R^i) + \sum_{\stackrel{1 \leq i \leq N}{j \neq i}} \int_0^T \int_{\R^d \times \R^d}  \frac 12\bigg(\vij(x-y)+\vji(x-y)\bigg)\dd e^t_\sharp Q^i \otimes e^t_\sharp Q^j  \dd t \\ 
& +&  \sum_i 
\int_{\Omega} g^i(\omega(T)) \dd Q^i,
\end{eqnarray*}
where $Q^i\in\Pb(\Omega)$ with $\Omega=\C([0,T];\R^d)$, $R^i$ is a well-chosen reference measure which will be defined in Section 3 and $\en(P\,\vert\,R)$ is the Boltzmann-Shannon entropy, that is
 \[ \en(\gamma\,|\,\pi)=\int \rho\log\rho \dd \pi,\;\text{if}\; \gamma = \rho \pi.\]
When the reference measure $\pi$ in the entropy is non indicated, is intended the Lebesgue measure. 

We denote the $N$-uple by $\bro:= (\rho^1, \dots, \rho^N)$.

The existence and uniqueness of weak solutions for this system will be discussed mainly in the full space $\R^d$, but all of our results can be easily adapted to the periodic setting of $\T^d$. For $\R^d$ we will need some condition at $\infty$ then we will use the spaces:
\begin{enumerate}
    \item[(i)] $\C_0^2([0,T]\times \R^d)= \{\phi \in \C^2 ( [0,T]\times \R^d) \ : \ \lim_{|x|\to +\infty} \phi (t,x)=0, \ \forall t\}$; \item[(ii)] $\mathrm{Lip} _0 (\R ^d)=\{u \in \mathrm{Lip}(\R^d) \ : \ \lim_{|x|\to +\infty} u (x)=0\}.$
\end{enumerate}
We will make the following assumptions. 
\begin{enumerate}
	\item[(A1)]  $\roi_0 \in \mathrm{Lip} _0 (\R ^d)$ (or $\roi _0\in \mathrm{Lip}(\T^d)$);
	\item[(A2)] $\en(\rho^i_0) < +\infty$; 
	\item[(A3)]  $g^i \in \C_0^2 (\R^d)$,
 \item[(A4)] $V^{i,j}*\rho \in \C_b (\R^d),$ for all $\rho$ such that $\sqrt{\rho} \in H^1 (\R^d).$
\end{enumerate}
\begin{rmk}\label{expass}
Assumption (A4) above is always satisfied if $V^{i,j}\in \C_b (\Rd)$ but also, for example, if $V^{i,j}(x)= \frac{1}{|x|^\alpha}$ with  $\alpha$ such that $ \frac{1}{|x|^\alpha} \in L^q_{loc}$  for a $\frac{d}{2d-2}\leq q\leq \infty$. So, if $d=3$, the Coulomb cost is allowed. 
The condition $\sqrt{\rho} \in H^1 (\R^d)$ may look unnatural, however, 
since $\rho$ is a probability measure, it corresponds to the fact that the Fisher information $I(\rho)$ is finite. Lemma 4.1 of \cite{BenCarDimNen}
shows that this is always the case in this setting. The relevance of the condition $\sqrt{\rho} \in H^1 (\R^d)$ is also given by the fact that such $\rho$'s are the electron densities associated to wave functions in quantum mechanics \cite{lieb1983density} and, in this respect, the Coulomb type cost discussed above is also relevant (notice that in \cite{lasry2007mean} the authors linked the MFG system with Hartree type equation in which case the Coulomb interaction arises quite naturally). 

Also, assumption (A4) is related to Th. \ref{heat}. Choosing to use duality with different functions spaces would, of course, requires different assumptions. 
\end{rmk}
The aim of this work is twofold. On the one hand, we introduce a non-local interaction term, which is not only interesting for the applications (see for instance \cite{cirant2015multi,cirant2017bifurcation,cardaliaguet2021introduction,barilla2021mean,bensoussan2018mean,dayanikli2023multi,sun2022mean} and the references therein) but also introduces a slight non-convexity. On the other hand, we try to strip naked the structure of the problem so to be as accessible as possible to non-specialists and students.

 Compared to \cite{BenCarDimNen} we allow  different kinds of non-locality and non-convexity in the interaction term and, also, aside with the periodic setting,  we consider problems on the entire space $\R^d$. Working on the entire space requires some care in the use of Fenchel-Rockafellar duality since the duality between spaces of measures and spaces of continuous functions is different. 
 Notice also that in \cite{BenCarDimNen} the authors mainly focus on the equivalence between the Eulerian and the Lagrangian formulation whereas, here, we also prove the equivalence between the Euler-Lagrange equations of the minimization problem and the MFG system. 
The separate convexity, which we use to obtain the necessary conditions for optimality, also appears in \cite{briani2018stable}. Non-local interaction of the type considered here also appears, for example, in \cite{bernardini2023ergodic}. Here we show that this kind of non-local interaction can be studied also in the present variational setting. 

 The paper is organized as follows: Section \ref{sec:variational} is devoted to introduce the Eulerian variational formulation (in the same flavor as \cite{lasry2007mean,cardaliaguet2015second}) as well as the analysis of minimizers and the duality. Notice here that since the interaction term is slightly non-convex we have to introduce and study the linearized functional. In Section \ref{sec:entropic_problem} we introduce the Lagrangian variational formulation which, as in \cite{benamouetalentropic}, turns out to correspond to a minimization of a relative entropy. Section \ref{sec:gamma_conv} is devoted to the time discretization of the two formulations and the $\Gamma-$convergence of the discrete problem to the continuous one. The time discretization, as well as the linearization functional introduced in section \ref{sec:variational} are then useful in order to introduce a suitable numerical method in section \ref{sec:numerical_appx} based on generalized Sinkhorn algorithm.   

\section{Variational formulation}
\label{sec:variational}
Let  $\Pb_2=\left\{\mu \in \Pb \ | \  \mu << \leb^d, \ \int |x|^2 \dd\mu <+\infty\right\}$ the space of probability measure which are absolutely continuous with respect to the $d$-dimensional Lebesgue measure and have finite second moment (we recall that the ambient space is either $\R^d$ or $\T^d$). We will consider the metric space $(\Pb_2, \W_2)$  constituted by $\Pb_2$ equipped with the Wasserstein distance $\W_2$.  This is a length metric space (see, for example, \cite{AmbGigSav}) and absolutely continuous curve in this metric space will play a role. Since the functional \eqref{eq:eulerian_1} is not convex in the couple $(\roi_t,v^i)$ we have to introduce the momentum variable $m_t^i=\roi_t v_t^i$ and
re-write the functional in the following formulation  where the first term (i.e. the kinetic energy term) is now jointly convex in the two variables (we will see later that this makes the linearized functional convex)
\begin{multline}
	\jc(\rho^1, m^1; \dots; \rho^N,m^N):= \sum_i \int _{[0,T]\times \Rd } \frac{|m_t^i|^2}{2\roi_t} \dd x \dd t +\\ \sum_{\stackrel{1\leq j \leq N}{j \neq i }}   \int_0^T \int_{\R^d \times \R^d} \frac 12\bigg(\vij(x-y)+\vji(x-y)\bigg) \dd \roi_t \otimes\dd \roj_t \dd t + \int_{\R^d} g^i(x) \dd \rho^i (T,x);
\end{multline}
if $\rho^j \in \Pb_2$,  $m^i << \rho^i$, with a little abuse of notations, $m^i_t= \dfrac{\partial m_t^i}{\partial \rho_t^i}$ denotes the Radon-Nikodym derivative of measures, and $+\infty$, otherwise.
In the previous problem, then,  $v^i= \dfrac{\partial m^i}{\partial \rho^i}.$ 

We will be interested in 
\begin{equation}\label{variational}
\inf \left\{  \jc(\rho^1, m^1; \dots; \rho^N, m^N)    \ | \ \
\partial_t \roi_t - \frac12 \Delta \roi_t + \Div m_t^i =0,\;\rho^i(0,x)=\roi_0,\;\forall i=1,\cdots,N\right\}.
\end{equation}	 

Absolutely continuous curves in $\Pb_2$  i.e. absolutely continuous curve in the metric space $(\Pb_2,\W_2)$ are, by now, well characterized. 
We report here part of the characterization which may be read, for example, in Th. 5.14 of \cite{SanOTAM}.
For a curve $\gamma$ in a metric space $X$ we denote by $|\gamma'|$ the metric derivative of $\gamma$. 
\begin{thm} Let $\rho_t$ be an absolutely continuous curve in $\Pb_2 (\R^d)$ (namely, $\rho_t\in AC([0,T];\Pb_2(\R^d))$) then for a.e. $t \in [0,T]$, there exists a vector field $v_t \in L^2_{\rho_t}$ such that 
\begin{equation}\label{trasport}
\partial_t  \rho_t + \Div (v_t\rho_t) =0,
\end{equation}
and for a.e. $t$ we have $\|v_t\|_{L^2_{\rho_t}} \leq |\rho'|(t)$, where $|\rho'|(t)$ denotes the metric derivative of $\rho$. 
Conversely, if $\rho_t$ is a curve in $\Pb_2 (\R^d)$ and for a.e. $t \in [0,T]$ there exists $v_t \in L^2_{\rho_t} (\R^d, \R^d)$ such that \eqref{trasport} holds, 
then $\rho_t$ is an absolutely continuous curve in $\Pb_2$ and for a.e. $t$ we have $|\rho'|(t)\leq \|v_t\|_{L^2_{\rho_t}}.$
\end{thm} 

If $\{\rho_t\}_{t \in [0,T]}$ is a curve of probability measures which solves 
\[\partial_t \rho_t - \frac{1}{2}\Delta \rho_t + \Div (v_t\rho_t)=0,  \]
with 
\[\int_0^T\int_{\R^d} |v_t|^2 \dd\rho_t \dd t < +\infty \]
and $\rho(0)=\rho_0$
then, by Lemma 4.1 of \cite{BenCarDimNen},  $\rho_t$ is absolutely continuous since for $w_t:= v_t - \frac12  \nabla \log \rho_t $ it holds 
\[\partial_t \rho - \Div (w_t \rho_t) =0,\]
and
\[\int _0^T \int |w_t|^2 \rho_t \dd x \dd t = \int _0^T \int |v_t|^2 \rho_t \dd x \dd t + \frac14 \int_0^T I(\rho_t) \dd t + \en (\rho(T)) - \en(\rho_0)<+\infty, \]
where
$I$ denotes the Fischer information, that is
\[I(\rho)=4||\nabla\sqrt{\rho}||^2_{L^2}\]
Actually Lemma 4.1 of \cite{BenCarDimNen} is richer and more detailed than we need and, for more, we refer the reader to the original paper.  We stress that the finiteness of the average kinetic energy associated to the curve $\rho_t$ and the vector field $v_t$ implies, by (iv) of Lemma 4.1 in \cite{BenCarDimNen}, the finiteness of the Fisher information and the entropy of $\rho(T,x)$.
In particular, if $(\rho_1, \dots, \rho_N)$ is admissible for problem \eqref{variational} then each component $\rho_i$ satisfies the assumptions above and so it is an absolutely continuous curve in $\Pb_2$. 

Existence of at least one minimiser of problem \eqref{variational}  will be proved in the next section.
It is well known that the first term of $\jc$ is not convex in the couple $(\rho^i,v^i)$.  However, as we have mentioned above, it can be rewritten as  
\[
\int_{[0,T]\times \Rd } f_2 (\roi_t, m^i_t) \dd x \dd t, 
\]
where  
\[
f_2(t,z) := \left\{
\begin{array}{ll}
	\frac{z^2}{2t} & \mbox{if} \ t> 0, \\
	0 & \mbox{if} 	\  t= 0, \mbox{and},\ z = 0.\\
	+ \infty & \mbox{otherwise.} 
\end{array}\right.\]
and the functional  $(\roi, m^i)\mapsto \int f_2 (\roi, m^i)$ is convex and lower semi-continuous (see, for example, Th. 5.18 of \cite{SanOTAM}). 

The second term is only separately convex in the $N$-tuple 
$(\rho^1, \dots, \rho^N)$ so, obtaining the optimality conditions goes through a directional linearisation process which will also be used in the following section.  

Let $(\overline \rho ^1, \overline m ^1; \dots ; \overline \rho ^N, \overline m ^N)$ be a minimiser of \eqref{variational}, define, for $i=1, \dots, N$
\begin{equation}\label{hi}
H^i(t,x)= \sum_{j \neq i} \int  (\vij (x-y)+\vji (x-y))  \dd\overline\rho ^j _t (y),
\end{equation}
and consider the functionals  
\[\jc^i(\rho, m) := \int_{[0,T]\times \Rd } \frac{|m_t|^2}{2\rho_t} \dd x \dd t + \int_{[0,T]\times \Rd } H^i \dd\rho_t \dd t + k^i + \int_{\R^d} g^i(x) \dd \rho^i (T,x), \]
where the constant is given by 
\begin{multline*}
	k^i := \sum_{j \neq i } \int_{[0,T]\times \Rd } \frac{|\overline m_t ^j|}{2\overline\rho_t ^j}^2\dd x \dd t  \\ + \int_{\R^d} g^j(x) \dd \overline \rho^j (T,x)+ \sum_{\stackrel{1\leq k \leq N}{j \neq k\neq i }}   \int_0^T \int_{\R^d \times \R^d} V^{j,k}( x-y) \dd \rho^k_t \otimes \dd\roj_t \dd t .
\end{multline*}	

The following proposition is slightly more than a remark. 
\begin{prop}\label{min1var}  For $i=1, \dots, N$, the couple $(\overline \rho ^i,\overline m^i)$ is a minimizer, among curves starting at $\rho_0^i$, of the functional 
	\[ \jc ^i (\rho,m) ,\] 
which is, moreover, a convex functional.
\end{prop}
\begin{proof} We observe that 
\[\jc^i(\rho,m)=\jc(\overline \rho ^1,\overline m^1; \dots; \rho,m; \dots;\overline  \rho ^N,\overline  m ^N) \geq  \jc ( \overline \rho ^1,\overline m^1; \dots;\overline  \rho ^N,\overline  m ^N) = \jc^i(\overline \rho^i_t, \overline m ^i).\]
The convexity of $\jc^i$ is due to the separate convexity of the only term of $\jc$ which is not convex.
\end{proof}
We will then make a careful analysis of a functional of the type of $\jc^i$. 
\subsection{Analysis of the minimizers}
In this subsection, for an $H$ which has the properties as the one defined as in \eqref{hi}, we will write necessary conditions for the minimization of 
\[\boxed{\fc(\rho, m):= \frac{1}{2}\int_{[0,T]\times \Rd } \frac{|m_t|^2}{\rho_t} \dd x dt + \int_{[0,T]\times \Rd } H \dd\rho_t \dd t + \int_{\R^d} g(x) \dd \rho (T,x),} \]
among solutions of 
\[\left\{\begin{array}{l}
\partial_t \rho_t - \frac12 \Delta \rho_t + \Div m_t =0;\\
\rho(0,x)=\rho_0(x).
\end{array}\right.\]

As frequently in this settings, the convex analysis, via the Fenchel-Rockafellar duality, will be the main tool. 
In the next two lemmas we will write $\fc$ as the sum of two convex conjugates. Since $(\rho,m)$ is couple of finite measures, the natural space for duality is the space of bounded continuous functions $\C_b([0,T]\times \R^d) \times\C_b([0,T]\times \R^d; \R^d).$  In the following will denote by $\C^{1,2}_b([0,T]\times \R^d)$ \sout{$\C^2_b([0,T]\times \R^d)$} the set of functions which are $\C^1_b$ in time and $\C^2_b$ in space. 

\begin{lem} \label{dualK1}  Let $(\alpha,\beta)\in\C_b([0,T]\times \R^d) \times\C_b([0,T]\times \R^d; \R^d)$, $H$ defined as in \eqref{hi} and 
\[\kc_1 (\alpha, \beta):= \left \lbrace 
\begin{array}{ll}
0 & \mbox{if} \ 	 \alpha + \frac{|\beta|^2}{2} \leq H, \\
+\infty, & \mbox{otherwise.}
\end{array}
\right.
\]	
Then 
\[\kc_1^* (\rho,m)=\left \{ \begin{array}{ll}
	\frac12\int \frac{|m|^2}{\rho}+\int H \dd \rho & \mbox{if}\ m << \rho,\\
	+ \infty & \mbox{otherwise}.
\end{array}\right.\]
\end{lem} 
\begin{proof} This is part of  Prop. 5.18 in \cite{SanOTAM}.
\end{proof}
\begin{lem}\label{dualK2}
Let $(\alpha,\beta)\in\C_b([0,T]\times \R^d) \times\C_b([0,T]\times \R^d; \R^d)$ and 
\[\kc_2 (\alpha, \beta)= \left\lbrace\begin{array}{ll}
	\int \phi(0,x)\rho_0 \dd x	& \mbox{if} \ \alpha = -\partial_t \phi -\frac{1}{2}\Delta \phi \ \mbox{and}\ \beta= -\nabla \phi,\\
	+\infty & \mbox{otherwise};
\end{array}	\right. \]
for some $\phi \in \C_b ^{1,2}([0,T]\times \R^d),$ such that $\phi(T,x) \geq -g(x).$
Then 
\[\kc_2 ^*(\rho, m)= \left\lbrace \begin{array}{ll}
	\int g(x) \rho(T,x)\dd x	& \mbox{if} \ \  \partial_t \rho -\frac{1}{2}\Delta \rho + \Div m =0,\ \rho(0)=\rho_0\\
	+\infty & \mbox{otherwise};
\end{array}	\right.\]	

\end{lem}
\begin{proof}
By the finiteness condition on $\kc_2$ we can write  
\begin{multline}\label{k2star}
	\kc_2^* (\rho, m)= \sup\{\langle -\partial_t \phi - \frac{1}{2}\Delta \phi , \rho \rangle  - \langle \nabla \phi , m \rangle - \int \phi (0,x) \rho_0(x) \dd x \ | \\
	  \phi \in \C_b ^{1,2}([0,T]\times \R^d),\ \phi(T,x) \geq -g(x). \}
	\end{multline}

Since we will need to integrate by parts, we build a suitable approximation of the admissible functions. To this aim we will use the fact that $\rho$, $m$ and $\rho_0$ are finite measures and then are tight. 

Let $\phi$ be admissible for \eqref{k2star}. To adjust the support we consider a smooth cut-off function $\chi_R $ such that $0 \leq \chi_R \leq 1$, $\chi_R \equiv 1 $ in $B(0,R)$ and  $\chi_R \equiv 0 $ in $\R^d \setminus B(0,2R)$, and such that the gradient and the Hessian of $\chi _R$ vanish, respectively, like $1/R$ and $1/R^2$.
We define $\phi_R:= \chi_R \phi$. $\phi_R$ satisfies the same regularity assumptions of $\phi$ but we can only guarantee $\phi(T,x) \geq -g(x)$ for $\|x\| \leq R$. By the tightness of $\rho$, $m$ and $\rho_0$, since $\phi_R$ is bounded with bounded derivatives and coincide with $\phi$  in $[0,T]\times B(0,R)$ as $R$ goes to $\infty$ the value of the functional converrges to the same value for $\phi$. Since $g \in \C^2_0$ we may define $\phi_R^\beta=\phi_R +\beta$
to obtain an admissible function for \eqref{k2star}. Since
\begin{multline*}
\langle -\partial_t \phi_R^\beta - \frac{1}{2}\Delta \phi_R^\beta , \rho \rangle  - \langle \nabla \phi_R^\beta , m \rangle - \int \phi_R^\beta (0,x) \rho_0(x) \dd x\\= \langle -\partial_t \phi_R - \frac{1}{2}\Delta \phi_R , \rho \rangle  - \langle \nabla \phi_R , m \rangle - \int \phi_R (0,x) \rho_0(x) \dd x -\beta,
\end{multline*}
We may let $R\to +\infty$ and $\beta \to 0$ in any order to obtain the corresponding value for $\phi$

Integrating by parts gives, 
\begin{multline*}
\langle \phi_R,  \partial_t \rho -\frac{1}{2}\Delta \rho + \Div m\rangle - \int \phi_R(T,x) \rho(T,x)  \dd x  + \int \phi_R(0,x) \rho (0,x) \dd x \\
-\int \phi _R (0,x) \dd \rho_0 (x) -\beta,
\end{multline*}
Letting first $R\to \infty$ and then $\beta \to 0$ we obtain for $\phi$, 
\begin{multline*}
\langle \phi,  \partial_t \rho -\frac{1}{2}\Delta \rho + \Div m\rangle - \int \phi(T,x) \rho(T,x)  \dd x + \int \phi(0,x) \rho (0,x) \dd x \\
-\int \phi  (0,x) \dd \rho_0 (x).
\end{multline*}
We consider, now,  the $\sup$ over suitable subsets of $\{ \phi \in \C_b ^{1,2}([0,T]\times \R^d),\ \phi(T,x) \geq -g(x) \}$. 
First, let $\hat\phi$ be admissible and such that
$\hat{\phi}(0,x)\equiv 0$ and $\hat{\phi}(T,x) >-g(x)$. For every $\varphi \in \C^\infty_c((0,T)\times \R^d))$ the function 
$\hat\phi + \varphi$ is admissible. Considering the $\sup$ over the (affine) space $\hat{\phi}+ \C^\infty_c$ gives that unless
\[  \partial_t \rho_t -\frac{1}{2}\Delta \rho_t + \Div m_t=0\]
in the sense of distributions the $\sup$ is $+\infty$  and then that unless 
\[\rho(0,x)=\rho_0\] 
the $\sup$ is $+\infty$. 
If these last two conditions are satisfied then 
\[\kc_2^* (\rho, m) = \int g(x) \rho(T,x) \dd x.\]

\end{proof}

By Fenchel-Rockafellar duality 
\begin{prop}\label{FenRock} Weak duality holds
\[\min \{\kc_1 ^*(\rho,m)+\kc_2 ^*(\rho,m)  \} \geq \sup \{-\kc_1 (\alpha, \beta)-\kc_2 (-\alpha, -\beta)  \}\]
\end{prop}
According to the expression of the functionals obtained above 
\begin{multline*}
	\min \{\kc_1 ^*(\rho,m)+\kc_2 ^*(\rho,m)  \} =\min \bigg\{  \frac12\int \frac{|m|^2}{\rho}+\int H \dd \rho + \int g(x) \rho(T,x) \dd x \\ \mbox{if}\ m<<\rho\ \mbox{and} \   \partial_t \rho -\frac{1}{2}\Delta \rho + \Div m=0,\}
\end{multline*}
which is our original problem. 
On the other hand the right-hand side is 
\begin{multline}\label{probmax}
\sup \bigg\{ \int \phi (0,x) \dd\rho_0 (x)  \ | \ \phi \in \C^{1,2}_b ([0,T]\times \R^d ), \\
-\partial_t \phi- \frac{1}{2}\Delta \phi + \frac{|\nabla \phi|^2}{2} \leq H, \ \phi(T,x) \leq  g(x). \bigg\}
\end{multline}

Problem \eqref{probmax} above may not have solution. To address this question we introduce 
\begin{multline*}
\mathcal{Z}=\{z \in \C_b ([0,T]\times \R^d) \ : \ z>0,\  \partial_t z- \frac{1}{2}\Delta z \geq- H(T-t,x) z \\ \mbox{in the sense of distributions,} \ z(0,x)\geq e^{-g}. \}
\end{multline*}
\[\A :=\{ \phi: [0,T] \times \R^d \to \R \ : \ z(t,x)=e^{-\phi(T-t,x)}\in \mathcal{Z}\}\]
the terminal condition $\phi (T,x) \leq  g(x)$ is encoded in the initial condition $z(0,x)\geq e^{-g(x)}.$
\begin{prop}
\label{prop:equivalence}
Since $H$ is continuous and bounded, the $\sup$ in \eqref{probmax} is equal to 
\begin{equation}\label{probmax2}
    \max_\A \int \phi (0,x) \dd\rho_0 (x).
\end{equation}
\end{prop} 
Since every admissible functions for \eqref{probmax}  belongs to $\A$, the inequality  \eqref{probmax} $\leq$ \eqref{probmax2} is immediate. We postpone the proof of the equivalence of the two problems and we turn
to the existence of a solution for \eqref{probmax2}, we have

\begin{thm}\label{heat}
There exists $\psi\in \A$ which is a distributional solution of 
\begin{equation}\label{hjreal}
\left\{\begin{array}{l}
	-\partial_t \psi- \frac{1}{2}\Delta \psi + \frac{|\nabla \psi|^2}{2} = H(t,x) ,\\
	\psi(T,x) = g(x)
\end{array}
\right.
\end{equation}
\mbox{and} \ 
\[
\int_{\mathbb R^d}\psi(0,x)\,d\rho_0(x)=\max_{\mathcal A}
\int_{\mathbb R^d}\phi(0,x)\,d\rho_0(x).
\]Then $\psi$ is a maximizer for problem \eqref{probmax2} above. Moreover $\psi \in \C_b ([0,T]\times \R^d)\cap L^\infty (0,T; W^{1,\infty}(\R^d))$.
\end{thm}
\begin{proof} Setting  $u(t,x):= \psi (T-t, x)$  \eqref{hjreal} is transformed in 
	\[
	\left\{\begin{array}{l}
		\partial_t u- \frac12 \Delta u + \frac{|\nabla u|^2}{2} = H(T-t,x) ,\\
		u(0,x) = g(x)
	\end{array}
	\right.
	\]
Then the Hopf-Cole transform $v(t,x)= e^{-u}$ gives the further simplification 
	\[
\left\{\begin{array}{l}
	\partial_t v- \frac12 \Delta v = -H(T-t,x) v ,\\
	v(0,x) = e^{-g(x)}
\end{array}
\right.
\]
So, setting $h(x,t):=-H(T-t,x)$ we can look at the solutions of the Heat equation with a zeroth order term.
By assumptions (A1), (A4), Remark \ref{expass} and Lemma 4.1 of \cite{BenCarDimNen}, $h$ is continuous and bounded. Since $g^i$ belongs to $\C^2_0 (\R^d)$ the initial datum $v(0,x)$ is bounded and Lipschitz.  The equation therefore admits a unique bounded mild solution $v$ which is represented via a semigroup and satisfies: 
\begin{itemize}
    \item $v \in \C_b ([0,T]\times \R^d);$
    \item $e^{-\|g\|_\infty-T\|H\|_\infty} \le v(t,x)\le e^{\|g\|_\infty+T\|H\|_\infty};$
    \item $v \in L^\infty (0,T; W^{1,\infty} (\R^d)).$
\end{itemize}
In particular, $0<c\le v(t,x)\le C.$ Reversing the path we set $u(t,x)=-\log v(t,x)$ and $\psi (t,x)= u(T-t,x)$.
Since \(v\) is Lipschitz in space
\[
\nabla u(t,x)= -\frac{\nabla v(t,x)}{v(t,x)}\ \mbox{and therefore} \ |\nabla u(t,x)|\le \frac{|\nabla v(t,x)|}{c}.
\]
Hence
\[
\|\nabla u\|_{L^\infty}
\le
c^{-1}\|\nabla v\|_{L^\infty}, \ \mbox{and then also}, \ \|\nabla\psi\|_{L^\infty((0,T)\times\mathbb R^d)}
\le c^{-1} \|\nabla v\|_{L^\infty((0,T)\times\mathbb R^d)}.
\]
By the regularity of the data, $v$ is a solution of the equation also in the distributional sense so that $\psi$ is a distributional solution 
of the HJ equation we are considering (notice that $\nabla \psi$ is, in particular, in $L^2_{loc}$), see, for example, Chapter 4 of \cite{lunardi2012analytic}. 
\end{proof}

\begin{proof}[Proof of prop. \ref{prop:equivalence}]
Let  $\psi$ be the maximizer of the relaxed dual problem \eqref{probmax2}. We need to find a sequence $\phi_n \in \C^{1,2}_b$ admissible for problem \eqref{probmax}, equibounded and such that $\phi_n(0,x) \to \psi (0,x)$ locally uniformly. Since $\rho_0$ is a probability measure (and the is tight) 
this will imply $\int \phi_n(0,x)\dd \rho_0\to \int \psi(0,x)\dd \rho_0 $ and then the equality of the values. 
We will actually produce a sequence that converges stronger than this so that we can use the same sequence in the next Theorem. 

 Consider the affine function $a(t)=-M-C(T-t)$ with $M$ large enough and $C>\|H\|_{\infty}$ so that  $a(t)$ satisfies the HJB inequality, $a(T)\leq g$ and $a\leq \psi$.
 We first extend $\psi$ in time setting $\psi (t,x)= \psi (0,x)\ \mbox{for}\ t\leq 0;$ and $\psi (t,x)= \psi (T,x)-\tilde{c} (T-t)=g(x)-\tilde{c} (T-t)\ \mbox{for}\ T \leq t$ with $\tilde{c} >\|H\|_\infty+\|\Delta g\|_{\infty}+\|\nabla g\|^2_{\infty} $. 
Let \(\chi\in C_c^\infty(\mathbb R^d)\) be such that
\[
0\le \chi\le1,\qquad
\chi\equiv1\hbox{ on }B_1,\qquad
\chi\equiv0\hbox{ on }\mathbb R^d\setminus B_2,
\]
and set \(\chi_R(x):=\chi(x/R)\). Define
\[
\psi_R(t,x):=\chi_R(x)\psi(t,x)+(1-\chi_R(x))a(t).
\]
Then \(\psi_R=\psi\) on \(B_R\), \(\psi_R=a\) outside \(B_{2R}\), and
\[
\psi_R(T,\cdot)\le g .
\]
Moreover, since \(\psi\) and \(\nabla\psi\) are bounded, and since $\nabla\chi_R$ and $\Delta\chi_R$ are, respectively, of the order of $\frac1R$ and $\frac{1}{R^2}$,
a direct computation gives  
\[
-\partial_t\psi_R-\frac12\Delta\psi_R
+\frac12|\nabla\psi_R|^2\le H+\varepsilon_R
\quad\hbox{in }\mathcal D'((0,T)\times\mathbb R^d),
\qquad
\varepsilon_R\to0 .
\]
We now regularize \(\psi_R\) slightly outside \([0,T]\) , and let \(\eta_\delta\) one-sided-in-time and space mollifier (kernel supported in negative times, so $\psi_{R,\delta}(t)$ only samples times in $(t,t+\delta)\subset(0,T+\delta)$). Set
\[
\psi_{R,\delta}:=\eta_\delta*\psi_R, 
\]
we obtain
\[
-\partial_t\psi_{R,\delta}-\frac12\Delta\psi_{R,\delta}
+\frac12|\nabla\psi_{R,\delta}|^2
\le
\eta_\delta*\bigg(-\partial_t\psi_R-\frac12\Delta\psi_R
+\frac12|\nabla\psi_R|^2\bigg)
\le
\eta_\delta*H+\varepsilon_R .
\]
Since \(\psi_R=a\) outside \(B_{2R}\), the only region where the mollification is
nontrivial is contained, for \(\delta<1\), in a fixed compact set depending on
\(R\). On this compact set \(H\) is uniformly continuous then there exists a modulus
\(\omega_R(\delta)\to0\) as \(\delta\downarrow0\) such that
\[
\eta_\delta*H\le H+\omega_R(\delta).
\]
Therefore
\[
-\partial_t\psi_{R,\delta}-\frac12\Delta\psi_{R,\delta}
+\frac12|\nabla\psi_{R,\delta}|^2
\le H+\varepsilon_R+\omega_R(\delta).
\]
 Since
\(\psi_R(T,\cdot)\le g\), \(g\) is uniformly continuous, and
\(\psi_R(t,\cdot)\to\psi_R(T,\cdot)\) locally uniformly as \(t\uparrow T\), we
have
\[
\gamma_{R,\delta}:=
\left\|\bigl(\psi_{R,\delta}(T,\cdot)-g\bigr)_+\right\|_\infty
\longrightarrow0
\qquad\hbox{as }\delta\downarrow0 .
\]
Define
\[
\phi_{R,\delta}(t,x):=
\psi_{R,\delta}(t,x)
-\bigl(\varepsilon_R+\omega_R(\delta)\bigr)(T-t)
-\gamma_{R,\delta}.
\]
Then \(\phi_{R,\delta}\in C_b^{1,2}([0,T]\times\mathbb R^d)\) and
\[
-\partial_t\phi_{R,\delta}-\frac12\Delta\phi_{R,\delta}
+\frac12|\nabla\phi_{R,\delta}|^2\le H,
\qquad
\phi_{R,\delta}(T,\cdot)\le g.
\]
Thus, $\phi_{R,\delta}$ is a competitor  for \eqref{probmax} and in particular we have
\[\eqref{probmax} \ge \int_{\R^d}\phi_{R,\delta}(0,x)\dd\rho_0(x).\]
For fixed \(R\), letting \(\delta\downarrow0\) gives
\[
\liminf_{\delta\downarrow0}
\int_{\mathbb R^d}\phi_{R,\delta}(0,x)\,\dd\rho_0(x)
\ge
\int_{\mathbb R^d}\psi_R(0,x)\,\dd\rho_0(x)
-T\varepsilon_R .
\]
This construction gives the following convergences: For fixed $R$ 
\[\psi_{R,\delta} \to \psi_R, \]
locally uniformly on $[0,T]\times \R^d$ and 
\[\nabla \psi_{R,\delta} \to \nabla \psi_R, \]
in $L^p([0,T]\times \R^d)$ for $1\leq p<+\infty$. The same, then, holds for $\phi_{R,\delta}$. 
Since $\psi_R=\psi$ on $B(0,R)$, passing to a diagonal subsequence we obtain
\[\phi_n \to \psi\]
locally uniformly and 
\[\nabla \phi_n \to \nabla \psi\]
in $L^p_{loc}$ for $1 \leq p< +\infty.$ 
Moreover, since both $\|\phi_n\|_\infty, \ \|\nabla \phi_n \|_\infty \leq C$, for any finite (then tight) measure 
$\mu$ on $[0,T]\times \R^d$, 
\[\nabla \phi_n \to \nabla \psi\]
in $L^p _\mu$ for $1\leq p <+\infty$. This approximation will be used in the proof of the next theorem.
\end{proof}

We conclude the section connecting the duality argument with the MFG system\eqref{mfg}, 
\begin{thm} Let $(\rho^1,m^1,...,\rho^N,m^N)$ be a minimizer for \eqref{variational}, define 
$v^i= \dfrac{\partial m^i}{\partial \rho^i_t},$ and $+\infty$, otherwise and, for all $i=1, \dots, N$, let $u^i$ be a maximizer for \eqref{probmax} with $H=H^i$ (defined in equation \eqref{hi}), $\rho_0=\rho^i_0$ and $g=g^i$. Then $v^i= -\nabla u^i$ and $(\rho^1,\dots,\rho^N, u^1, \dots, u^N)$ is a solution of \eqref{mfg}.
\end{thm}
\begin{proof}
Fix \(i\in\{1,\ldots,N\}\), and let \(u^i\) be a maximizer of the dual
problem \eqref{probmax} (we change the notation since we are back to the case of $N$ variables). By Theorem \ref{heat},
\[
-\partial_tu^i-\frac12\Delta u^i+\frac12|\nabla u^i|^2=H^i
\]
in the sense of distributions, with terminal condition
\(u^i(T,\cdot)=g^i\).

Moreover,
\[
u^i\in \mathcal C_b([0,T]\times\mathbb R^d),\qquad
\nabla u^i\in L^\infty((0,T)\times\mathbb R^d).
\]

Since the vector field \(-\nabla u^i\) is bounded, there exists a unique
distributional solution \(\tilde\rho^i\) of
\[
\partial_t\tilde\rho^i
-\frac12\Delta\tilde\rho^i
-\operatorname{div}(\tilde\rho^i\nabla u^i)=0,
\qquad
\tilde\rho^i(0)=\rho_0^i.
\]
Setting
\[
\tilde m^i:=-\nabla u^i\,\tilde\rho^i,
\]
the pair \((\tilde\rho^i,\tilde m^i)\) is admissible for the primal
problem.

Since \((\rho^i,m^i)\) is a minimizer, Proposition \ref{FenRock} gives
\begin{multline}
\label{slack2}
\int_0^T\!\!\int_{\mathbb R^d}
\frac{|\tilde m^i|^2}{2\tilde\rho^i}
+
\int_0^T\!\!\int_{\mathbb R^d}
H^i\,\dd\tilde\rho_t^i
+
\int_{\mathbb R^d}
g^i\,\dd\tilde\rho_T^i
\\
\ge
\int_0^T\!\!\int_{\mathbb R^d}
\frac{|m^i|^2}{2\rho^i}
+
\int_0^T\!\!\int_{\mathbb R^d}
H^i\,d\rho_t^i
+
\int_{\mathbb R^d}
g^i\,\dd\rho_T^i
\\
\ge
\int_{\mathbb R^d}
u^i(0,x)\,\dd\rho_0^i(x).
\end{multline}

Consider now the approximation \(u_n^i\in C_b^{1,2}\) of $u^i$ constructed in the previous proposition and, in particular consider the convergence of $\nabla u_n^i$ for $p=2$ and $\mu=\leb^1 \times \tilde\rho_t^i$. 

Applying the weak formulation of the equation satisfied by $(\tilde \rho_i, \tilde m _i)$ with test
function \(u_n^i\), we obtain
\begin{multline}
\int_{\mathbb R^d}
u_n^i(T)\,\dd\tilde\rho_T^i
-
\int_{\mathbb R^d}
u_n^i(0)\,\dd\rho_0^i
\\
=
\int_0^T\!\!\int_{\mathbb R^d}
\left(
\partial_tu_n^i
+\frac12\Delta u_n^i
\right)
\,\dd\tilde\rho_t^i
+
\int_0^T\!\!\int_{\mathbb R^d}
\nabla u_n^i\cdot \dd\tilde m_t^i .
\end{multline}

Passing to the limit as \(n\to\infty\), using the above convergences,
the identity
\[
\tilde m^i=-\nabla u^i\,\tilde\rho^i,
\]
and the Hamilton--Jacobi equation satisfied by \(u^i\), yields
\[
\int_0^T\!\!\int_{\mathbb R^d}
H^i\,\dd\tilde\rho_t^i
+
\int_{\mathbb R^d}
g^i\,\dd\tilde\rho_T^i
=
\int_{\mathbb R^d}
u^i(0)\,\dd\rho_0^i
-
\frac12
\int_0^T\!\!\int_{\mathbb R^d}
|\nabla u^i|^2\,\dd\tilde\rho_t^i.
\]

Also from $\tilde m^i=-\nabla u^i\,\tilde\rho^i,$
we have
\[
\frac{|\tilde m^i|^2}{2\tilde\rho^i}
=
\frac12|\nabla u^i|^2\,\tilde\rho^i,
\]
and therefore
\[
\int_0^T\!\!\int_{\mathbb R^d}
\frac{|\tilde m^i|^2}{2\tilde\rho^i}
+
\int_0^T\!\!\int_{\mathbb R^d}
H^i\,d\tilde\rho_t^i
+
\int_{\mathbb R^d}
g^i\,\dd\tilde\rho_T^i
=
\int_{\mathbb R^d}
u^i(0)\,\dd\rho_0^i.
\]

Combining this identity with \eqref{slack2}, all the inequalities become
equalities. Since the function
\[
(m,\rho)\longmapsto \frac{|m|^2}{2\rho}
\]
is strictly convex in \(m\), equality in Fenchel's inequality implies
\[
m^i=-\rho^i\nabla u^i,
\]
that is,
\[
v^i=\frac{\partial m^i}{\partial\rho^i}=-\nabla u^i.
\]
This concludes the proof.
\end{proof}

\section{The Entropic problem}
\label{sec:entropic_problem}
In this section we focus on a Lagrangian formulation to \eqref{variational} based on a minimization of a relative entropy. In particular by proving the existence of a solution to the Lagrangian formulation we deduce the existence also for the Eulerian one.  
Notice that the existence of a minimizer for the Eulerian formulation is usually obtained by means of Fenchel-Rockafellar theorem (on the $d-$dimensional torus), see, for instance, \cite{cardaliaguet2015second,cardaliaguet2015weak}. One usually exploits the convexity of the problem, but in this case the interaction term between the populations makes the functional non-convex. Here, we prove existence by using the relationship between the minimizers of the Lagrangian and Eulerian formulation. We don't know of any approach via direct methods of the Calculus of Variations to the Eulerian formulation.\\
Let $\Omega:= (\C([0,T], \R^d), \| \cdot \|_\infty)$ and let $R^i \in \Pb(\Omega) $  be Wiener measures defined by
$$R^i=\int_{\R^d} \roi_0 \delta_{x+B^i (t)(x)} \dd x$$ 
where the $B_i$ are the classical Brownian motions.
For $Q^i\in \Pb(\Omega)$ consider the following variational problem to which we will refer as Lagrangian
 
\begin{equation}\label{var}
\min\{ J(Q^1, \dots, Q^N) \ : \ e^0_\sharp Q^i= \rho_0^i \}
\end{equation}
where 
\begin{eqnarray*}
J(Q^1, \dots, Q^N) &:= & \sum_i \en(Q^i | R^i) + \sum_{\stackrel{1 \leq i \leq N}{j \neq i}} \int_0^T \int_{\R^d \times \R^d} \frac 12\bigg(\vij(x-y)+\vji(x-y)\bigg) \dd e^t_\sharp Q^i \otimes \dd e^t_\sharp Q^j  \dd t \\ 
& +&  \sum_i 
\int_{\R^d} g^i(x) \dd e^T_\sharp Q^i,
\end{eqnarray*}
where $e^t_\sharp :\Pb (\Omega) \to \Pb (\R^d)$ is the evaluation map at time $t$; furthermore, notice that for every $t\in [0,T]$, $e^t_\sharp$  is continuous for the narrow convergences.

\begin{rmk}\label{finitepoint} Taking $Q^i=R^i$ gives at least one point at which $J$ is finite. In fact, in that case, the entropy terms are equal to 
$0$. For the other terms, it holds:
\[\int_0^T \int_{\R^d \times \R^d} V(x-y) \dd e^t_\sharp R^i \otimes \dd e^t_\sharp R^j  \dd t = \int_0^T \int_{\R^d \times \R^d} V(x-y)  (p_t*\roi_0) (x) ( H_t *\roj_0 ) (y) \dd x \dd y   \dd t\]
where $H_t$ is the Gaussian kernel, that is
\[H_{t}(z):=\dfrac{1}{(2\pi t)^d}\exp{\bigg(-\dfrac{|z|^2}{2t}\bigg)},\;t>0,\; z\in\R^d.\]

This integral is finite by assumption (A4).
Analogously, 
\[ \int_{\R^d} g^i(x) \dd e^T_\sharp R^i=\int_{\R^d} g^i(x) (H_T*\roi_0)(x) \dd x , \]
which is also finite by the assumptions on $g^i$ and $\roi_0$.
\end{rmk}

\begin{thm} The minimum in \eqref{var} is finite and there exists a minimizer $(\overline{Q}^1, \dots, \overline{Q}^N)$.
\end{thm}
\begin{proof} 
By Remark \ref{finitepoint} above the infimum of $J$ is finite.  
Let $\{(Q^1_n, \dots,Q^N_n)\}_{n \in \N}$ be a minimizing sequence, so that $J(Q^1_n, \dots, Q^N_n) \leq C$. 
Since $V$ is bounded from below we have
\begin{eqnarray*}
\sum_i \en(Q^i_n | R^i) + {\bm k} +  \sum_i  \int_{\Omega} g^i(\omega(T)) \dd Q^i_n 
+ \max_\varphi  \{ \sum_i \int _{\R^d}  \varphi (0,x) \dd \roi_0 - \int_{\Omega} \varphi (0, \omega(0)) \dd Q^i_n \}\\
\leq \sum_i \en(Q^i | R^i) + \sum_{\stackrel{1 \leq i,j \leq N}{j \neq i}} \int_0^T \int_{\R^d \times \R^d} \frac12\bigg(\vij(x-y)+\vji(x-y)\bigg) \dd e^t_\sharp Q^i \otimes e^t_\sharp Q^j  \dd t \\ 
+  \sum_i  \int_{\Omega} g^i(\omega(T)) \dd Q^i + \max_\varphi  \{ \sum_i \int _{\R^d}  \varphi (0,x) \dd \roi_0 - \int_{\Omega} \varphi (0, \omega(0)) \dd Q^i \}\\
\leq C,
\end{eqnarray*}
 where we remind that $\Omega$ is the set of continuous paths $\C([0,T],\Rd)$.
The above inequality implies, since the other addenda of the first term have linear growth 
$$ \en(Q^i_n | R^i) \leq C_1, \ \ \forall i. $$ 
Thus (see appendix), up to subsequences, there exists $Q^i$ with 
$$ Q^i_n \stackrel{*}{\rightharpoonup} Q^i$$ 
and since all the terms constituting $J$ are lower semicontinuous, $(Q^1, \dots, Q^N)$ is a minimizer. 

The first term of $J$ is the entropy which is lower-semicontinuous (see Lemma 9.4.3 of \cite{AmbGigSav} and the appendix), the second and the third terms are given by continuous functionals composed with 
the continuous operator $Q \mapsto e^t_\sharp Q$ finally the last term is the $\sup$ of continuous functionals and l.s.c. as such. 
\end{proof}

From now on, we denote $Q^i_t=e^t_\sharp Q^i$ $\forall t\in[0,T]$.

The next theorem from \cite{BenCarDimNen} relates the Lagrangian problem \eqref{var} to the  Eulerian formulation  of system \eqref{mfg} that we discussed in the previous section. This will allow us to obtain the existence of a minimizer $(\overline{\rho} ^1, \overline{m}^1 \dots, \overline{\rho} ^N, \overline{m}^N)$ for problem \eqref{variational}. Since only the first term of the energy in \eqref{variational} depends on $m$,  
given a curve $\rho \in \C([0,T]; (\Pb_2, \mathcal W _2)) $ we define 
\[\ec(\rho):= \inf_v \left\{\frac12 \int_0^T\int_{\Rd} |v_t|^2 \dd\rho_t \dd t \ : \ \partial_t \rho_t - \frac12 \Delta \rho_t + \Div (v_t \rho_t)=0\right\},\]
and then, we decompose the minimization \eqref{variational} in a two steps minimization writing it as 
\begin{multline}\label{twosteps}
	\inf \bigg\{  \sum_i \ec(\rho^i_t)+ \sum_{\stackrel{1\leq j\leq N}{j \neq i }}   \int_0^T \int_{\R^d \times \R^d} \frac 12\bigg(\vij(x-y)+\vji(x-y)\bigg) \dd \roi_t \otimes\dd \roj_t \dd t \\ + \int_{\R^d} g^i(x) \dd \rho^i (T,x) \ : \ \rho^i \in \C( [0,T]; (\Pb_2, \mathcal W _2)),\;\rho^i(0,x)=\roi_0,\;\forall i=1,\cdots,N\bigg\}.
\end{multline}
This last problem only depends on the curve $\rho$ since the role of $m$ (or $v$) has already been encoded in $\ec$. 
In an analogous way we can decompose also the minimization problem \eqref{var} as
\begin{multline}\label{twosteps_entropic}
	\inf \bigg\{  \sum_i \mathcal S(\rho^i)+ \sum_{\stackrel{1\leq j\leq N}{j \neq i }}   \int_0^T \int_{\R^d \times \R^d} \frac 12\bigg(\vij(x-y)+\vji(x-y)\bigg) \dd \roi_t \otimes\dd \roj_t \dd t \\ + \int_{\R^d} g^i(x) \dd \rho^i (T,x) \ : \ \rho^i \in \C( [0,T]; (\Pb_2, \mathcal W _2)),\;\rho^i(0,x)=\roi_0,\;\forall i=1,\cdots,N\bigg\},
\end{multline}
where
\begin{equation}
	\label{EntCost}
	\mathcal S(\mu^i)=\inf\{ \mathcal H(Q^i|R^i)\;|\;Q^i_t=\mu^i_t,\;\forall t\in [0,T], \}
\end{equation}
so obtaining a problem which depends only on the curve $\mu^i$. 
Then using the following result of \cite{BenCarDimNen} we have the existence of a minimizer of \eqref{variational} from that of a minimizer of \eqref{var}.
\begin{prop}[see corollary 4.7 of \cite{BenCarDimNen}]\label{equiva} If $\rho_0 \in \Pb_2$ and $\mathcal H(\rho_0)< +\infty$ then 
\[ \mathcal S (\rho)= \ec (\rho) + \mathcal H (\rho_0).\]
\end{prop}

\begin{thm}\label{existence} There exists minimizers of problems \eqref{twosteps} and \eqref{variational}.
\end{thm}
\begin{proof} If $\overline{Q}=(\overline{Q}^1,\dots,\overline{Q}^N)$ is a minimizer of problem \eqref{var}, the curve $\overline{Q}_t:=(\overline{\rho}^1,\dots,\overline{\rho}^N)$ is a minimizer of problem \eqref{twosteps_entropic}. By Proposition \ref{equiva} above, problems \eqref{twosteps} and \eqref{twosteps_entropic} have the same minimizers (while the minimal values differs by a constant). So $(\overline{\rho}^1,\dots,\overline{\rho}^N)$ is also a minimizer for problem \eqref{twosteps}. If we choose, then, for each $i\in \{1,\dots, N\}$, $v^i$ a vector field as in the definition of $\ec (\overline{\rho}^i)$ and consider $\overline{m}^i:= v^i\overline{\rho}^i$ we have that $(\overline{\rho}^1,\overline{m}^1, \dots,\overline{\rho}^N,\overline{m}^N)$ is a minimizer for problem \eqref{variational}.
\end{proof}

\section{Time discretization and $\Gamma$-convergence}
\label{sec:gamma_conv}

Before introducing a suitable discretizations we shortly recall the two equivalent problems we studied above. 
The main player in those problems is a vector curve of probability measures 
\[\rho \in (\C ([0,T]; (\Pb_2, \W_2) ))^N. \]
Given a positive integer $K$ the discrete version of the $\rho$ above is a $N$-tuple of $(K+1)$-vectors of probability measures 
\[\robf _K\in \Pb_2^{K+1} \times \dots \times  \Pb_2^{K+1} .  \]
So the $i$-th components $\robf^i_0,\cdots,\robf_K^i$ of $\robf^i_K$ is a discrete version of the $i$-th curve. 
To this $(K+1)-$tuple of probability measures we can associate the piece-wise constant curve 
\[\ropc^i(t)= \robf^i_j ,  \ \ \ \mbox{for}\ t \in \bigg[\frac{jT}{K}, \frac{(j+1)T}{K}\bigg).
\]
The ambient space for the Eulerian versions of the problems was 
\[\ac_0 = \{\rho:[0,T]\to (\Pb_2 )^N \ | \ \mbox{abs. cont. and s.t.}\  \rho(0)=\rho_0\}.\] 
We may first introduce the discretized space
\[\ac_0^K:= \{\robf_K  \in (\Pb_2 ^{(K+1)})^N \ : \ \robf_0^i=\rho^i_0 , \  i=1,\dots, N\}\]

We say that $\robf_ K \to \rho$ as $K\to +\infty$ if  for all $i=1, \dots, N$
\[ \sup_{t \in [0,T]} \W_2 (\ropc^i(t), \rho^i (t) ) \to 0.\] 

Given $N$ continuous curves of measures $\rho^i\in \mathcal C([0,T],(\mathcal P_2(\mathbb R^d),\mathcal W_2))$, that is $\rho^i:t\in [0,T]\mapsto \rho^i_{t}\in\mathcal P_2(\mathbb R^d)$, we defined
 the minimal energy  $\ec(\rho)$, the minimal entropic cost $\mathcal{S} (\rho)$, 
as well as the cost
\[ \mathcal F(\rho^1,\cdots,\rho^N)= \sum_{\stackrel{1 \leq i,j \leq N}{j \neq i}} \int_0^T \int_{\R^d \times \R^d} \frac 12\bigg(\vij(x-y)+\vji(x-y)\bigg) \dd \rho^i_t \otimes \dd\rho^j_t  \dd t \\ 
+  \sum_i  \int_{\R^d} g^i(x) \dd\rho^i(T,x).\]
Thus, the minimization problem \eqref{variational} and \eqref{var} can now be re-written in the following way

\begin{equation}\label{varEul}
\inf\{ \sum_{i=1}^N \ec(\rho^i)+\mathcal F(\rho^1,\cdots,\rho^N)\;|\; \rho \in \ac_0\}, 
\end{equation}
and
\begin{equation}
    \label{var2}
    \inf\{ \sum_{i=1}^N \mathcal S(\rho^i)+\mathcal F(\rho^1,\cdots,\rho^N)\;|\; \rho \in \ac_0\}.
\end{equation}


We define the time discretization of \eqref{varEul} as
\begin{equation}
	\label{varDiscEul}
	\inf\{\sum_{i=1}^N \mathcal E^K(\robf^i)+F^K(\robf^1,\cdots,\robf^N)\;|\; 
	\robf_K \in\mathcal T^K\},
\end{equation}
where
\[ \mathcal E^K(\robf^i):=\sum_{k=0}^{K-1}\mathcal E_{\frac{T}{K}}(\robf^i_k,\robf^i_{k+1}), \]
with 
\[\mathcal E_{\frac{T}{K}}(\mu,\nu):=\inf\{ \dfrac{1}{2}\int_0^{\frac{T}{K}}\int_{\mathbb R^d}|v_t|^2\dd\rho_t \dd t\;|\;\dt\rho-\dfrac{1}{2}\Delta\rho_t+\Div(\rho_t v_t)=0,\;\rho_0=\mu,\; \rho_{\frac{T}{K}}=\nu\}, \]
and 

\[ \mathcal T^K:=\{(\robf^1,\cdots,\robf^N)\in ((\mathcal P_2(\mathbb R^d),\mathcal W_2)^{(K+1)})^N\;|\;\robf_0^i=\rho_0^i,\;\forall i=1,\cdots,N \}. \]
In a similar way one can discretize in time the Lagrangian counterpart

\begin{equation}
    \label{EntDisc}
    \mathcal S^K(\robf^i):=\inf\{\mathcal H(Q^i|R^i)\;|\;Q\in\mathcal P(\Omega),\; Q^i_{j\frac{T}{K}}=\robf^i_j,\;j=0,\cdots,K \}
\end{equation}
as well as the interaction term and the final cost
\begin{equation}
    \label{CostDisc}
    \mathcal F^K(\robf^1,\cdots,\robf^N)=\dfrac{T}{K}\sum_{k=1}^{K-1}\sum_{\stackrel{1 \leq i,j \leq N}{j \neq i}} \int_{\R^d \times \R^d}\frac 12\bigg(\vij(x-y)+\vji(x-y)\bigg) \dd \robf^i_k \otimes\dd \robf^j_k+ \sum_{i=1}^N  \int_{\R^d} g^i(x) \dd\robf_K^i,
\end{equation}
where, we recall, $\robf^i$ stands now for a vector of measures, that is $\robf^i\in\Pb_2(\mathbb R^d)^{(K+1)}$ which discretize a curve of measures. 
Notice that \eqref{EntDisc} can be equivalently reformulate as a classical multi-marginal problem; that is for $i=1,..,N$ we have
\begin{equation}
    \label{EntDiscStat}
    \mathcal S^K(\robf^i):=\inf\{\mathcal H(\pi^i_K|R^i_K)\;|\;\pi^i_K\in\Pi(\robf_0^i,\cdots,\robf^i_K)\},
\end{equation}
where $\Pi(\robf_0^i,\cdots,\robf^i_K)$ is the set of probability measures on $(\mathbb R^d)^{K+1}$ having $\robf_0^i,\cdots,\robf^i_K$ as marginals and
\[R^i_K:=(e^{0},e^{\frac{T}{K}},\cdots,e^T)_\sharp R^i. \]
Then the discretized \eqref{var2} takes the form
\begin{equation}
    \label{varDisc}
    \inf\{\sum_{i=1}^N \mathcal S^K(\robf^i)+\mathcal F^K(\robf^1,\cdots,\robf^N)\;|\; (\robf^1,\cdots,\robf^N)\in\mathcal T^K\}.
\end{equation}

\begin{thm} As $K\to +\infty$,
\[\sum_{i=1}^N \ec^K (\robf^i) + \mathcal F^K(\robf^1,\dots, \robf^n)  \stackrel{\Gamma}{\to} \sum_i^N \ec(\rho^i)+ \mathcal{F} (\rho^1, \dots, \rho^N). \]
\end{thm}
\begin{proof} {\bf $\Gamma-\limsup$ inequality}
Let $\rho\in \ac_0$ be such that 
\[\sum_i \ec (\rho^i) + \mathcal{F} (\rho) < +\infty. \]
We consider the discretization $\robf_K$ of $\rho$ given by $\robf _j= \rho(\xi_j^K) $ where the times $\xi_j^K \in \bigg[j\frac{T}{K}, (j+1)\frac{T}{K}\bigg)$ for $j=1,\dots, K-1$ are chosen according to Remark \ref{sumint} below, $\xi_K^0=0$ and $\xi_K^K= T$ . 
Since $t \mapsto \rho_t$ is uniformly continuous, the convergence, 
\[\robf_K \to \rho,\]
is verified. 
We check the convergence of each term of the functional, starting with $\mathcal F ^K$.
Since $t \mapsto \rho_t$ is $\W_2$ continuous, it is $w*$ continuous and since $(x,y)\mapsto V(x-y)$ is lower semi-continuous, the same holds for 
\[t \mapsto \sum_{\stackrel{1 \leq i,j\leq N}{i\neq j}} \frac 12\bigg(\vij(x-y)+\vji(x-y)\bigg) \dd\rho^i_t(x) \otimes \dd\rho^j_t (y).\]
And this last map can be considered as $g$ in Remark \ref{sumint}.
By the choice of discretizing times
\[\int g^i(x) \dd\robf_k^i = \int g^i(x) \dd\rho^i(T,x). \]
Concerning the energy term, it is enough to study the convergence for each $i$. In the addendum $\ec_{\frac{T}{K}}(\robf^i_k,\robf^i_{k+1})$ we may take $\rho\bigg(t-k\frac{T}{K}\bigg)$  and the corresponding optimal vector field $v$ as test function so obtaining
\[\ec^K (\robf^i_K)=\sum_{k=0}^{K-1} \ec_{\frac{T}{K}}(\robf^i_k,\robf^i_{k+1})\leq 
\ec (\rho^i),\]
which, passing to the $\limsup$ as $K \to +\infty$ concludes the proof of the $\Gamma- \limsup$ inequality. 

{\bf $\Gamma-\liminf$ inequality} Let $\rho\in \ac_0$ and let $\robf_K \to \rho.$
We have 
\[\frac{T}{K}\sum_{k=0}^{K-1} \sum_{\stackrel{1 \leq i,j\leq N}{i\neq j}}\int \frac 12\bigg(\vij(x-y)+\vji(x-y)\bigg) \dd \robf_k^i \otimes \dd\robf_k^j = \int_0^T \sum_{\stackrel{1 \leq i,j\leq N}{i\neq j}} \int \frac 12\bigg(\vij(x-y)+\vji(x-y)\bigg) \dd \ropc_t^i \otimes \dd\ropc_t^j \dd t.\]
The very definition of convergence $\robf_K \to \rho$ implies that $\ropc_t^i \otimes \ropc_t^j \stackrel{*}{\rightharpoonup} \dd\rho_t^i\otimes\dd\rho_t^j$ for all $t,i,j$.
Since $V$ is lower semi-continuous 
\[\mu\otimes \nu \mapsto \int V(x-y) \dd \mu \otimes\dd\nu, \]
is $w*$ lower semi-continuous, by the lower semi-continuity of the integral 
\[\liminf_{K\to +\infty} \int_0^T \sum_{\stackrel{1 \leq i,j\leq N}{i\neq j}} \int \frac 12\bigg(\vij(x-y)+\vji(x-y)\bigg) \dd \ropc_t^i \otimes\dd \ropc_t^j \dd t.\]
Concerning the last addendum of the functional 
\[\int g^i(x) \dd\robf_K^i \to \int g^i(x) \dd \rho^i,\]
since $g^i$ is continuous and bounded by assumptions. So let us look at the energy term $\ec^K$. This term may be studied as in \cite{BenCarDimNen} by usual methods in Calculus of Variations.
Also in this case it is enough to study $\ec^K(\robf^i)$. Since we may assume that 
the total energy of $\robf_K$ is bounded, the same holds for $\ec^K(\robf^i)$.
Let $\roam_K,\ \vam_K$ be an almost minimizer for $\ec^K(\robf^i)$, i.e. 
$\roam_K :[0,T]\to \Pb_2$
is absolutely continuous, $\roam_K \bigg(j\frac{T}{K}\bigg)= \robf^i_j$ for $j=0,\dots,K$
$\vam_{K,t} \in L^2_{\roam_t}$,
\[\partial_t \roam_K -\frac{1}{2}\Delta \roam_K + \Div \vam_K \roam_K=0\]
and
\[C \geq\ec^K (\robf_K ^i) \geq \int _0^T \int \frac{|\vam_k|^2}{2}\dd\roam_{K,t} \dd t -\frac{1}{K}.\]
We estimate the total variation of the measure  $\mam_K:= \vam_K\roam_K$ in $[0,T]\times \R^d$.
Using 
\[C \geq \int _0^T \int \frac{|\mam_K|^2}{2 \roam_K} \dd x\dd t,\]
and the Young's inequality we obtain 
\[\|\mam_K\| _{[0,T]\times \R^d} \leq C+ \|\roam\|_{[0,T]\times \R^d}.\]
Since $\{\roam_K\}$ is also bounded in total variation we have that, up to subsequences 
\[(\roam_K, \mam_K) \stackrel{*}{\rightharpoonup} (\rho, m), \]
with $(\rho,m)$ solution of 
\[\partial_t \rho_t -\frac{1}{2}\Delta \rho_t + \Div m_t=0.\]
Since $m=v \rho$ for a suitable $v$ and applying the lower semi-continuity Theorem 2.34 of \cite{ambrosio2000functions}, originally from \cite{buttazzo1991functionals}, we have,
\begin{multline*}
\liminf_{K\to +\infty} \ec^K (\robf_K ^i) \geq 
\liminf_{K\to +\infty}\int _0^T \int \frac{|\vam_k|^2}{2}\dd\roam_{K,t} dt = \\ \liminf_{K\to +\infty}\int _0^T \int \frac{|\mam_k|^2}{2\roam_{K,t}}\dd x\dd t 
\geq \int _0^T \int \frac{|m|^2}{2\rho_t}\dd x\dd t \geq \ec(\rho^i).
\end{multline*}
\end{proof}

\begin{rmk}\label{sumint} Let $g:[0,T ] \to [0,+\infty)$ be a lower-semicontinuous functions and let $K\in \N$. There exists points $\xi_K^j \in \bigg[j\frac{T}{K}, (j+1)\frac{T}{K}\bigg)$ such that
\[\frac{T}{K} \sum_{j=1}^{K-1}g(\xi_K^J) \to \int_0^T g(t)\dd t. \]

\end{rmk}




\section{Numerical Approximation}
\label{sec:numerical_appx}
We now present a numerical scheme in order to solve the discretized in time problem \eqref{varDisc}. In particular the scheme is based on a variant of the Sinkhorn algorithm, successfully used to solve many variational problems involving optimal transport \cite{CuturiSinkhorn,benamouetalentropic,benamou2016numerical,benamou2017generalized,peyregradientflows,peyre2017computational,Chizat,thesislulu} and it is an adaptation of the scheme introduced in \cite{BenCarDimNen,barilla2021mean}. More numerical schemes for multi-population MFG have beed developed, for example, in \cite{carlini2018fully} based on some previous works of the same authors.

 We recall that \eqref{varDisc} reads as:
\[ \inf\left\{\sum_{i=1}^N \mathcal S^K(\robf^i)+\mathcal F^K(\robf^i,\cdots,\robf^i)\;|\; (\robf^i,\cdots,\robf^i)\in\mathcal T^K\right\}\]
where $\mathcal S^K$ is itself defined by   \eqref{EntDiscStat} which is an entropy minimization with multi-marginal constraints.

Denoting    $P^k:(\R^d)^{K+1}\rightarrow(\R^d)$  the $k-$th canonical projection we can obviously rewrite \eqref{varDisc}  as  an optimization problem over plans $\pi^i_N$ only:
\begin{multline}
\label{primalEnt}
 \inf \left\{ \sum_{i=1}^N \mathcal H(\pi^i_K\vert R_K^i) +i_{\rho^i_0} (P_\#^0\pi^i_k)+G(P^K_{\#}\pi^i_K)\right.\\
 \left.+\mathcal F^K(\bm{P_{\#}\pi^1_K},\cdots,\bm{P_{\#}\pi^N_K} )\;:\;(\pi^i_K)_{i=1}^N\in (\mathcal P((\R^d)^{K+1}))^N \vphantom{\frac12}\right\} ,     
\end{multline} 
 where $\bm{P_{\#}\pi^i_K}=(P^0_{\#}\pi^i_K,\cdots,P^K_{\#}\pi^i_K)$, $G$ is the final cost and is of the form $G(\rho)=\int g\dd\rho$, and 
 \[i_{\rho_0}(\rho)= \begin{cases} 0 \mbox{ if $\rho=\rho_0$}\\ +\infty \mbox{ otherwise} \end{cases}\]
  is the indicator function in the convex analysis sense and is used to enforce the initial  condition.
  We recall that for all $i=1,\cdots,N$, the static reference measure  $R^i_K$ in the relative entropy term is defined as follows
  \[R^i_K:=(e^{0},e^{\frac{T}{K}},\cdots,e^T)_\sharp R^i. \]
Moreover, since we are considering the reversible Wiener measure, it turns out that $R^i_K$ can be decomposed by using the heat kernel as
\[R^i_K:=\bigg(\prod_{k=1}^K H_{\frac{T}{K}}(x_k-x_{k-1}) \bigg)\dd\rho_0(x_0),\cdots \dd x_K, \]
where \[H_{t}(z):=\dfrac{1}{(2\pi t)^d}\exp{\bigg(-\dfrac{|z|^2}{2t}\bigg)},\;t>0,\; z\in\R^d,\]
and $|\cdot|$ denotes the standard euclidean norm.
We also need a discretization in space, for instance we use a
$M$ grid points to discretize $\R^d$, then $\pi^i_K$ and 
$R_K^i$ become tensors in $\R^{MN}$. For sake of simplicity 
we will keep the continuous in space notation, but form now
on integral must be understood as finite sums and
$x_0,\cdots,x_K$ as $M$ vectors.
Notice that thanks to the euclidean norm, the heat kernel 
$H_t(z)$ can be decomposed as a product along the  dimension of the one dimensional kernel,
that is \[H_t(z)=\prod_{j=1}^dh_t(z_j),\]
where $h_t(z_j)$ is the heat kernel in dimension one. This implies that, instead 
of storing a matrix $H_t\in\R^{M\times M}$, one can just store
$d$ small matrices belonging to $\R^{\sqrt[d]{M}\times\sqrt[d]
{M}}$.
One can now try to generalize the algorithm introduced in \cite{Chizat} and its multi-marginal variant \cite{benamouetalentropic} to the multi-population case in the same flavour as \cite{barilla2021mean}. However, since the interaction term between populations is non-convex, it happens that we are out of the domain of application of Sinkhorn algorithm. 
A way to overcome this difficulty is through a semi-implicit approach in order to treat the interaction term, that is at step $n+1$ the $i-$th plan $\pi^{i,(n+1)}_K$ is computed as the optimal solution of a linearized problem obtained by injecting the $j-$th, with $j\neq i$, plans $\pi^{j,(n)}_K$ computed at the previous step: for all $i=1,\cdots,N$
\begin{multline}
\label{linearized_Sink}
 \pi^{i,(n+1)}_K:=\argmin_{\pi^i_K\in \mathcal P((\R^d)^{K+1})} \left\{ \mathcal H(\pi^i_K\vert R_K^i) +i_{\rho^i_0} (P_\#^0\pi^i_k)+G(P^K_{\#}\pi^i_K)\right.\\
  \left.+\mathcal F^K_i(\bm{P_{\#}\pi^{i}_K})\vphantom{\frac12}\right\}  ,
\end{multline}
where 
\[\mathcal F^K_i(\bm{P_{\#}\pi^{i}_K}):= \mathcal F^K(\bm{P_{\#}\pi^{1,(n+1)}_K},\cdots,\bm{P_{\#}\pi^{i}_K},\cdots,\bm{P_{\#}\pi^{N,(n)}_K} )\]
We now have to solve $N$ finite-dimensional strictly convex minimization problems. 
Then, for each problem strong duality holds and \eqref{linearized_Sink} cab be re-written as follows
\begin{multline}
    \label{lin_dual_Sink}
    \sup_{(u_0^i,\cdots,u_K^i)}-F_0^*(-u^i_0)-G^*(-u_K^i)-\mathcal F^*_i(-u^i_1,\cdots,-u^i_{K-1})\\
    -\int\bigg (\exp{(\oplus_{j=0}^Ku_j^i)}-1 \bigg)R_K^i,
\end{multline}
where with a slight abuse of notation $F^*_i$ denotes the sum of the Fenchel-Legendre transform of each term in $\mathcal F_i^K$.
Denoting by $\pi^{i,(n+1)}_K$ and $u^{i,(n+1)}_j$ the optimal solution to \eqref{linearized_Sink} and \eqref{lin_dual_Sink}, respectively, it follows that the unique solution to \eqref{linearized_Sink} has the form
\[ \pi^{i,(n+1)}_K(\bm x):=\bigg(\otimes_{k=0}^Ke^{u^{i,(n+1)}_k(x_k)}\bigg)R^i_K(\bm x),  \]
where $\bm x=(x_0,\cdots,x_K)$.
\begin{rmk}[Structure of the optimal solution]
By definition of the linearized term $\mathcal F_i^K$ it follows that the $\mathcal F^*_i$ is just a sum of indicator function in the convex analysis sense; this implies that for all $k=1,\cdots,K-1$
\[u^{i,(n+1)}_k(x_k)=\sum_{j\neq i}\int V^{i,j}(x_k-y_k)\rho^{j,(n)}_{k}(y_k)\dd y_k,\]
where $\rho^{j,(n)}_{k}(y_k):=P_{k,\sharp}\pi_K^{j,(n)}$ is the marginal of the solution computed at the previous step. In the same way if the final cost if of the form $G(\rho)=\int g\dd\rho$ then $u^{i,\star}_K(x_K)=g^i(x_K)$. For sake of clarity we consider always these kinds of functional, even if the algorithm can be defined with more complex functional (we refer the reader to \cite{benamouetalentropic,barilla2021mean} for some examples with 1 or 2 populations; the extension to $N$ populations is then straightforward).
\end{rmk}
Notice that thanks to the remark above the generalised Sinkhorn algorithm  takes now the following form. \\
\begin{algorithm}
    \caption{Multi-population Sinkhorn}\label{algo}
    \begin{algorithmic}[1]
     \Require $u^{i,(0)}_k=0$
     \While{$\sum_{i=1}^N||\rho_0^{i,(n)}-\rho^{i}_0||<$tol}
     \For{$i=1 : N$} 
     \For{$k=0:K$}
     \If{$k=0$}
    \State $u_0^{i,(n+1)}=\log(\rho^i_0)-\log\bigg(\int\bigg (\exp{(\oplus_{j=1}^Ku_j^{i,(n)})} \bigg)R_K^i\bigg)$
    \ElsIf{$k\neq 0,K$}
    \State $u_K^{i,(n+1)}=\sum_{j\neq i}\int V^{i,j}(x_k-y_k)\rho^{j,(n)}_{k}(y_k)dy_k$
    \ElsIf{$k=K$}
    \State $u_K^{i,(n+1)}=g^i$
    \EndIf
    \EndFor
     \EndFor
     \EndWhile
    \end{algorithmic}\label{alg: pseudo-code}
\end{algorithm}

In the following numerical results we  take a space $M\times M$ discretization of $[0,1]^2$ with $M=100$ and a time discretization of $[0,1]$ with $K=32$ time step.
Let us, firstly, consider the 2 densities case: we have always considered the same initial data
\[\rho^1_0=\exp(-50(x^1-.2)^2-50(y^1-.5)^2),\quad \rho^2_0=\exp(-50(x^2-.8)^2-50(y^2-.5)^2)\] 
and the same final costs
\[ g^1=50((x^1-0.8)^2+(y^1-0.45)^2),\quad g^2=50((x^2-0.2)^2+(y^2-0.5)^2),\]
such that we expect the two Gaussians to switch position.
As for the interaction potential we have considered in Figure \ref{fig:support_strongV} a strong repulsion given by $V(x,y)=120\chi_{||x-y||<0.2}(x,y)$ and in Figure \ref{fig:support_truncatedCoulomb} a truncated Coulomb repulsion $V(x,y)=\min\bigg(1000,\frac{1}{||x-y||}\bigg)$. Notice in both cases the effect of entropic term  which oblige the densities to spread but at the same time the effect of the repulsive interaction forbid them to touch each other (the distance between them depends on the kind of the repulsion term).
\newline
\begin{figure}[h!]
    \TabFour{
    \includegraphics[width=.25\linewidth]{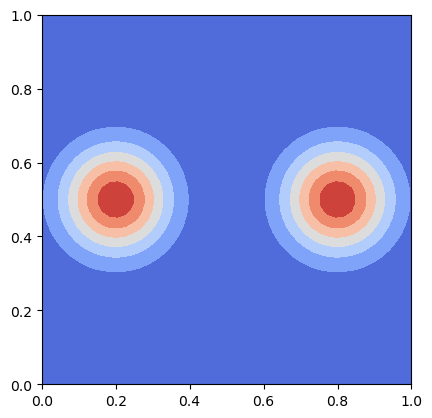}&\includegraphics[width=.25\linewidth]{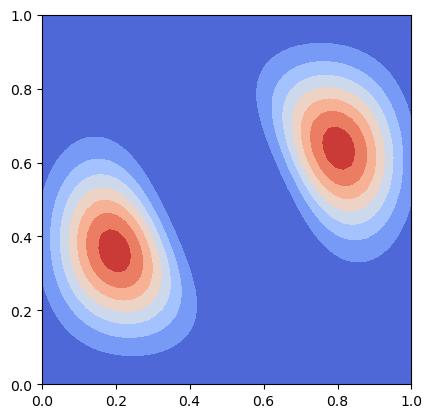}&\includegraphics[width=.25\linewidth]{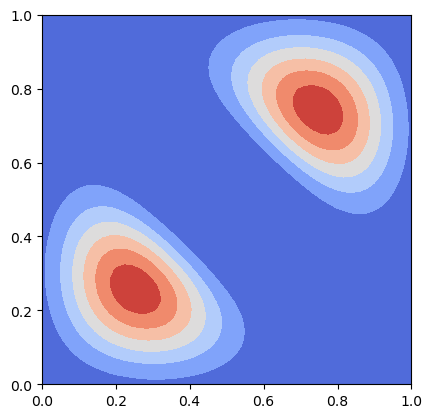}&\includegraphics[width=.249\linewidth]{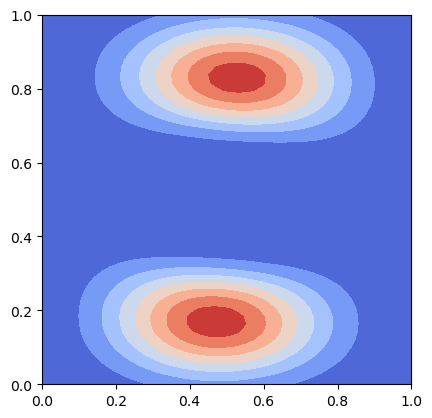}\\
    $t=0$&$t=1/8$&$ t=1/4$&$t=1/2$\\
    \includegraphics[width=.25\linewidth]{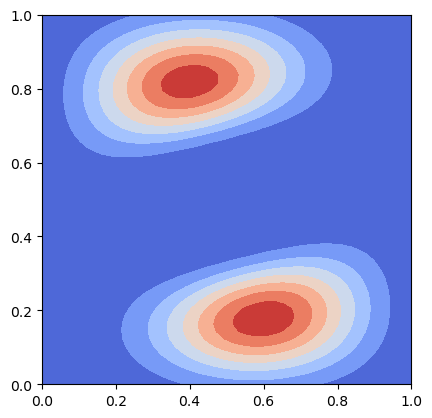}&\includegraphics[width=.25\linewidth]{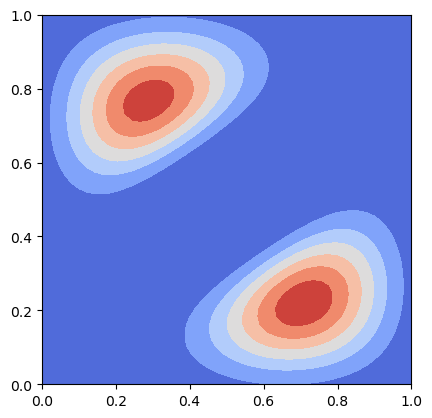}&\includegraphics[width=.25\linewidth]{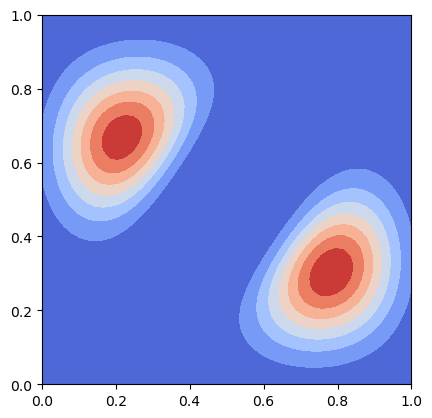}&\includegraphics[width=.249\linewidth]{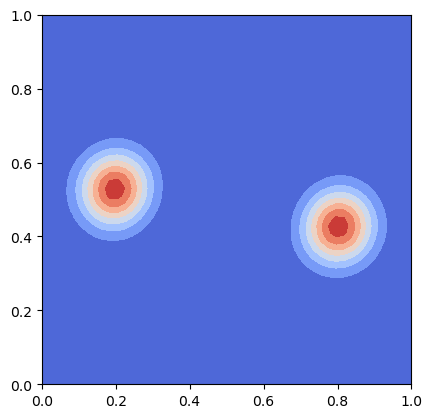}\\
    $t=5/8$&$t=3/4$&$ t=7/8$&$t=1$}

    \caption{Support of $\rho^1$ and $\rho^2$ for $V(x,y)=120\chi_{||x-y||<0.2}(x,y)$. }
    \label{fig:support_strongV}
\end{figure}
\begin{figure}[h!]
    \TabFour{
    \includegraphics[width=.25\linewidth]{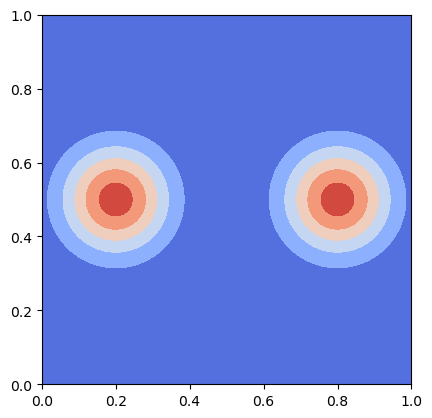}&\includegraphics[width=.25\linewidth]{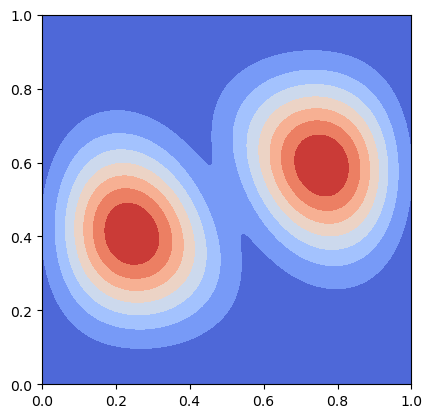}&\includegraphics[width=.25\linewidth]{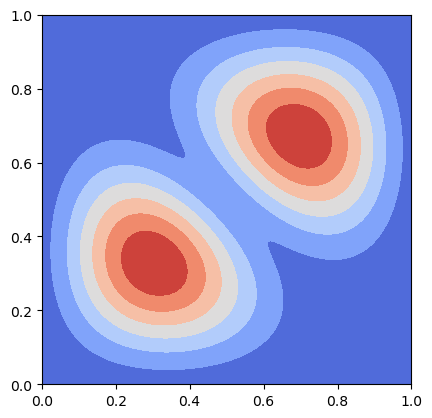}&\includegraphics[width=.249\linewidth]{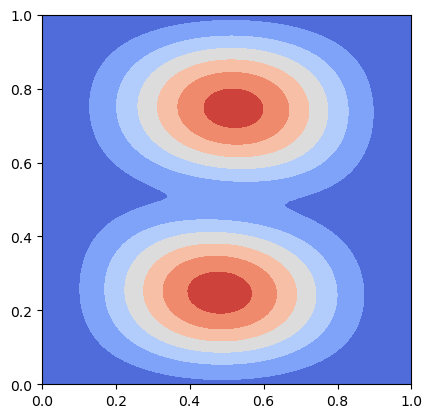}\\
    $t=0$&$t=1/8$&$ t=1/4$&$t=1/2$\\
    \includegraphics[width=.25\linewidth]{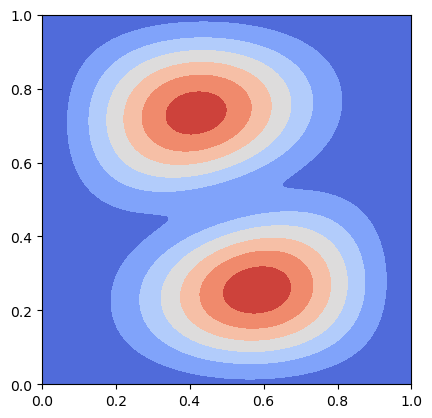}&\includegraphics[width=.25\linewidth]{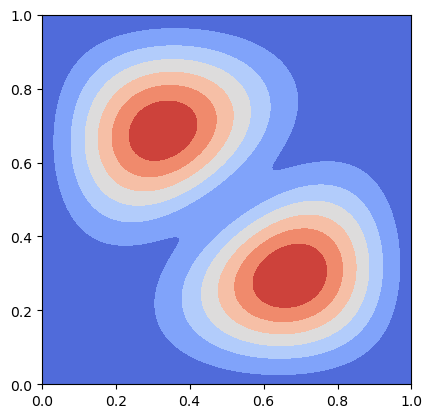}&\includegraphics[width=.25\linewidth]{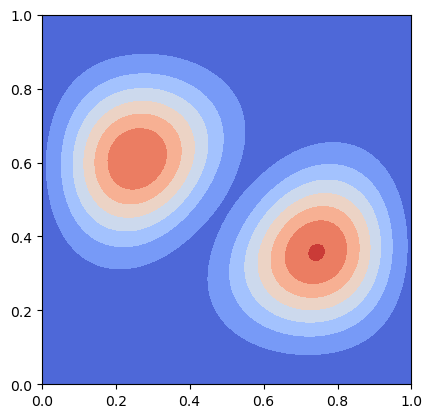}&\includegraphics[width=.249\linewidth]{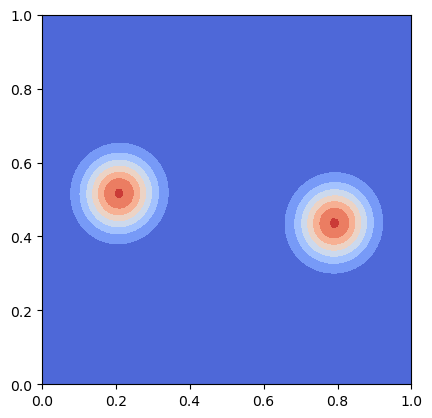}\\
    $t=5/8$&$t=3/4$&$ t=7/8$&$t=1$}

    \caption{Support of $\rho^1$ and $\rho^2$ for  $V(x,y)=\min\bigg(1000,\frac{1}{||x-y||}\bigg)$. }
    \label{fig:support_truncatedCoulomb}
\end{figure}

It is now straightforward to extend the  theory and the numerical method to a slightly more general model with a viscosity parameter $\eps$. The Mean Field Game system
(\ref{mfg}) now  takes the form~:

\begin{equation}\label{mfgeps}
\left\{ \begin{array}{l}
-\dt \ui-\frac{\varepsilon}{2} \Delta \ui +\frac{1}{2} |\nabla \ui|^2 = \sum_{j \neq i}\int_{\R^d} \vij(x-y)\roj_t \dd y, \\
\dt \roi_t -\frac{\varepsilon}{2} \Delta \roi_t -\Div (\nabla \ui \roi_t)=0,\\
\roi(0,x)=\roi_0(x),\;  \ui(T,x)= g^i(x);\\
\end{array}\right.
\end{equation}
The Lagrangian formulation \eqref{var} we have proposed becomes
\begin{equation}\label{vareps}
\min\{ J_{\varepsilon}(Q^1, \dots, Q^N) \ : \ e^0_\sharp Q^i= \rho_0^i \}
\end{equation}
where 
\begin{eqnarray*}
J_\varepsilon(Q^1, \dots, Q^N) &:= & \sum_i \en(Q^i | R_\varepsilon^i) + \sum_{\stackrel{1 \leq i \leq N}{j \neq i}} \int_0^T \int_{\R^d \times \R^d} V(x-y) \dd e^t_\sharp Q^i \otimes \dd e^t_\sharp Q^j  \dd t \\ 
& +&  \sum_i 
\int_{\R^d} g^i(x) \dd e^T_\sharp Q^i.
\end{eqnarray*}
where $R^i_\varepsilon$ are the reversible Wiener measure induced by a Brownian motion with variance $\varepsilon$.
Notice that one could choose different $\varepsilon$ for each population.
In particular, if we discretize the problem in time, we have that the reference measure $R^i_{\eps,K}$ can be still decomposed by using the heat kernel $H_{\eps t}(z)$.

Notice that we can still use the algorithm we have introduced above, but the performance, in terms of iterations to converge, will be affected by small values of $\varepsilon$. At least formally, when the viscosity is small, \eqref{vareps} is an approximation of the following Lagrangian formulation of first-order variational mean-field games (see  \cite{benamou2017variational} for the one population case)
\begin{equation}\label{firstordervareps}
\min\{ \mathcal K(Q^1, \dots, Q^N) \ : \ e^0_\sharp Q^i= \rho_0^i \}
\end{equation}
where 
\begin{eqnarray*}
\mathcal K(Q^1, \dots, Q^N) &:= & \sum_i K(Q_i) + \sum_{\stackrel{1 \leq i \leq N}{j \neq i}} \int_0^T \int_{\R^d \times \R^d} V(x-y) \dd e^t_\sharp Q^i \otimes \dd e^t_\sharp Q^j  \dd t \\ 
& +&  \sum_i 
\int_{\R^d} g^i(x) \dd e^T_\sharp Q^i.
\end{eqnarray*}
where 
\begin{equation}\label{defdeK}
K(Q):=  \frac{1}{2}  \int_{\Omega}  \int_0^T \vert \dot \omega(t)\vert^2  \mbox{d}t \mbox{d} Q(\omega).
\end{equation}
This also  implies that we can use the Sinkhorn algorithm, with small $\varepsilon$, in order to approximate the solution to first-order MFGs. In Figures \ref{fig:support_strongVeps} and \ref{fig:support_truncatedCoulombeps} we have considered the same data as above but with $\varepsilon=0.005$; notice now the effect of a weaker diffusion term which prevents the densities from spreading.
\begin{figure}[h!]
    \TabFour{
    \includegraphics[width=.25\linewidth]{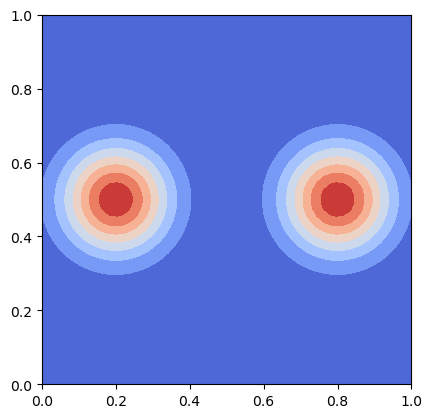}&\includegraphics[width=.25\linewidth]{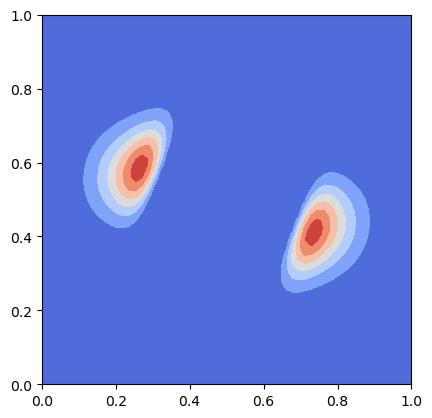}&\includegraphics[width=.25\linewidth]{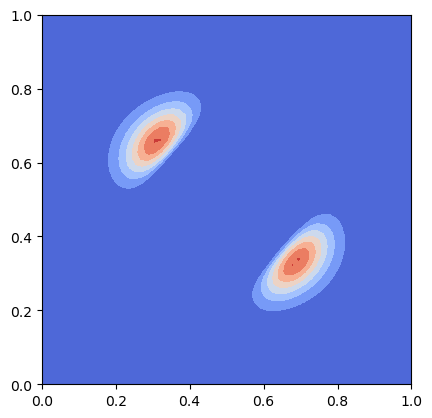}&\includegraphics[width=.249\linewidth]{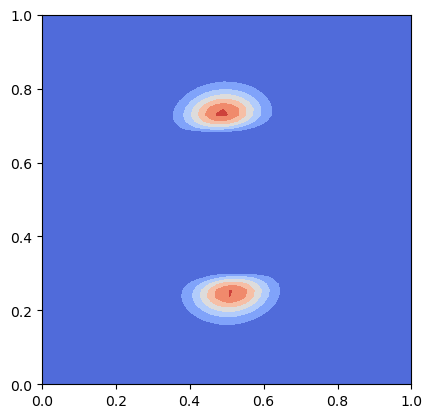}\\
    $t=0$&$t=1/8$&$ t=1/4$&$t=1/2$\\
    \includegraphics[width=.25\linewidth]{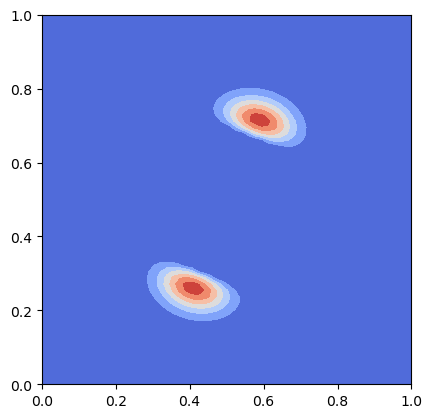}&\includegraphics[width=.25\linewidth]{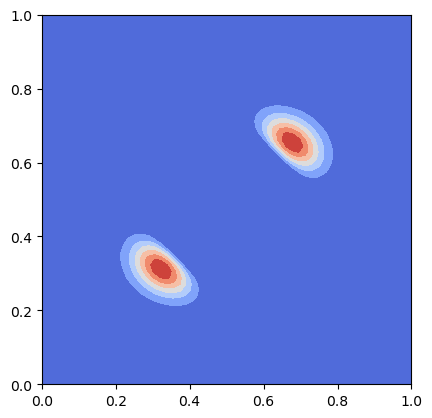}&\includegraphics[width=.25\linewidth]{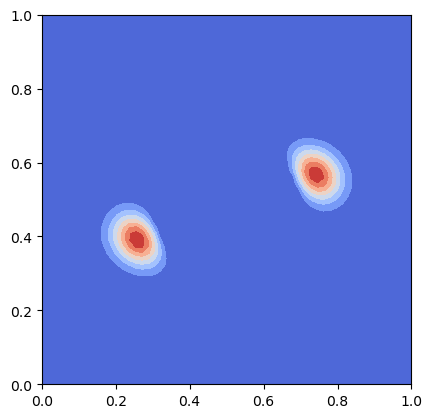}&\includegraphics[width=.249\linewidth]{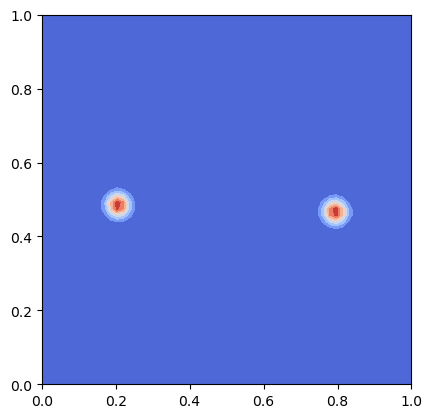}\\
    $t=5/8$&$t=3/4$&$ t=7/8$&$t=1$}

    \caption{Support of $\rho^1$ and $\rho^2$ for $\varepsilon=.005$ and $V(x,y)=120\chi_{||x-y||<0.2}(x,y)$. }
    \label{fig:support_strongVeps}
\end{figure}
\begin{figure}[h!]
    \TabFour{
    \includegraphics[width=.25\linewidth]{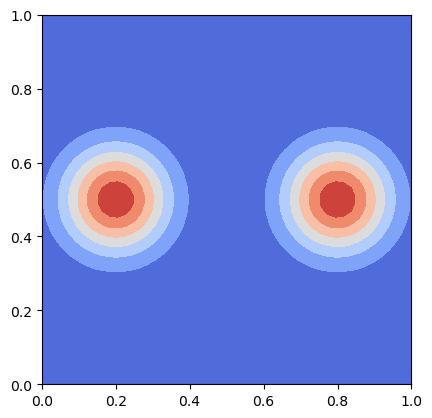}&\includegraphics[width=.25\linewidth]{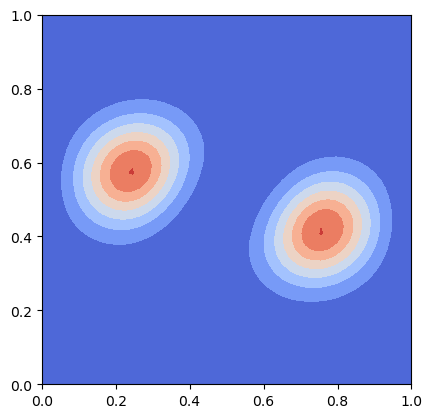}&\includegraphics[width=.25\linewidth]{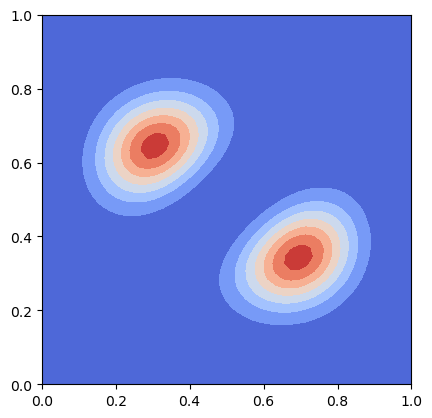}&\includegraphics[width=.249\linewidth]{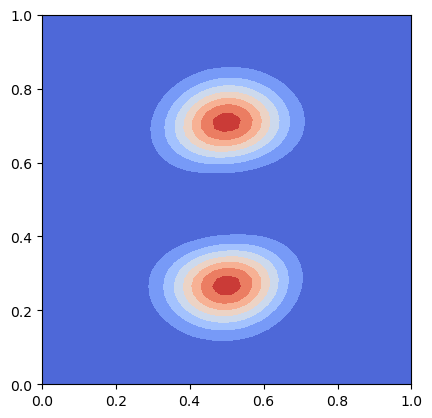}\\
    $t=0$&$t=1/8$&$ t=1/4$&$t=1/2$\\
    \includegraphics[width=.25\linewidth]{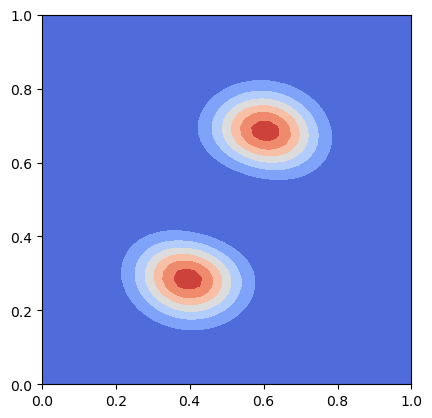}&\includegraphics[width=.25\linewidth]{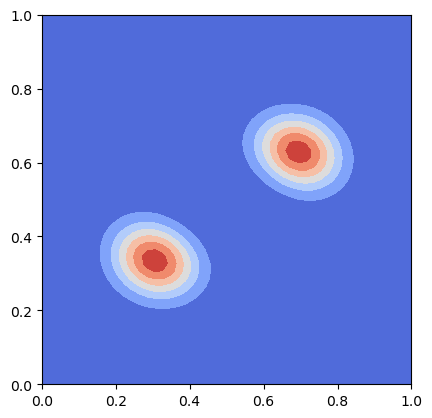}&\includegraphics[width=.25\linewidth]{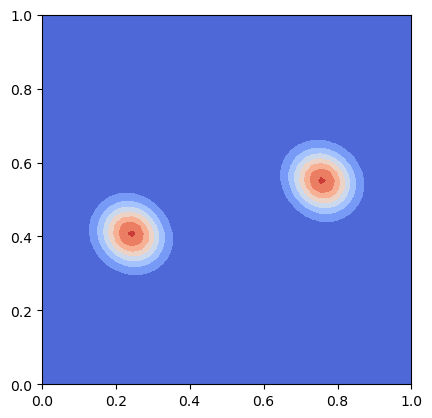}&\includegraphics[width=.249\linewidth]{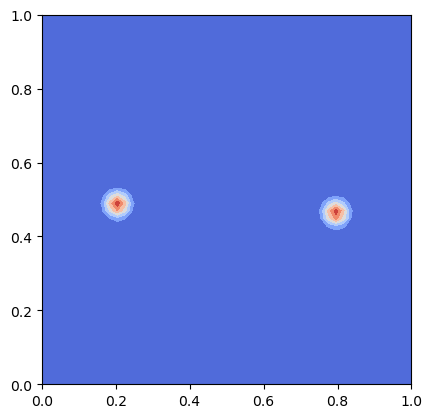}\\
    $t=5/8$&$t=3/4$&$ t=7/8$&$t=1$}

    \caption{Support of $\rho^1$ and $\rho^2$ for $\varepsilon=.005$ and $V(x,y)=\min\bigg(1000,\frac{1}{||x-y||}\bigg)$. }
    \label{fig:support_truncatedCoulombeps}
\end{figure}
Finally, we consider a 3 populations example with initial data
\[\rho^1_0=\exp(-50(x^1-.2)^2-50(y^1-.5)^2),\; \rho^2_0=\exp(-50(x^2-.8)^2-50(y^2-.5)^2),\]
\[\rho^3_0=\exp(-80(x^3-.5)^2-80(y^3-.1)^2),\] 
and final costs
\[ g^1=50((x^1-0.8)^2+(y^1-0.5)^2),\; g^2=50((x^2-0.5)^2+(y^2-0.1)^2),\;g^3=50((x^2-0.2)^2+(y^2-0.5)^2)\]
which induces a rotation of the populations. In Figure \ref{fig:support_strongV_3} we plot the evolution of  support  of the 3 densities considering as interaction term the strong repulsion we have taken above.
\begin{figure}[h!]
    \TabFour{
    \includegraphics[width=.25\linewidth]{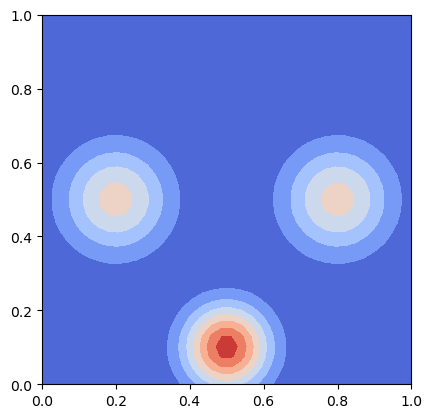}&\includegraphics[width=.25\linewidth]{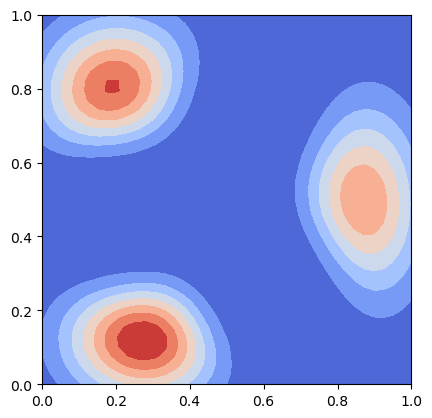}&\includegraphics[width=.25\linewidth]{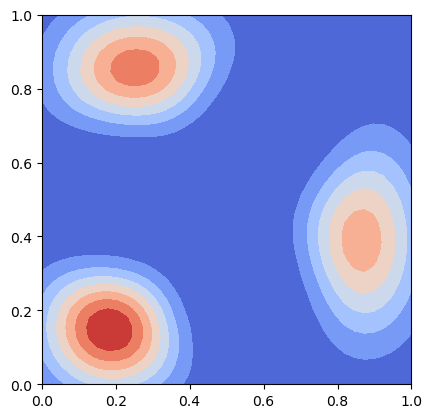}&\includegraphics[width=.249\linewidth]{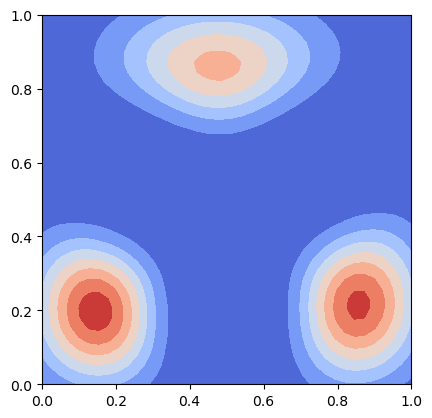}\\
    $t=0$&$t=1/8$&$ t=1/4$&$t=1/2$\\
    \includegraphics[width=.25\linewidth]{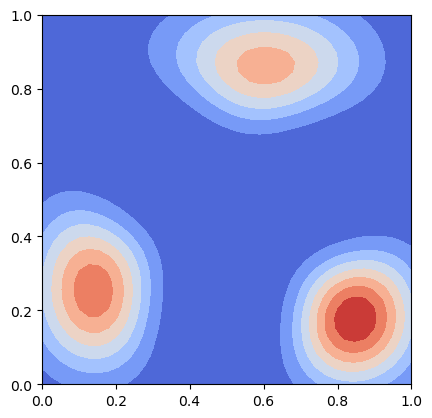}&\includegraphics[width=.25\linewidth]{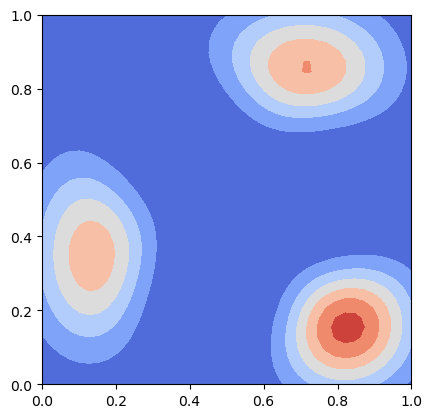}&\includegraphics[width=.25\linewidth]{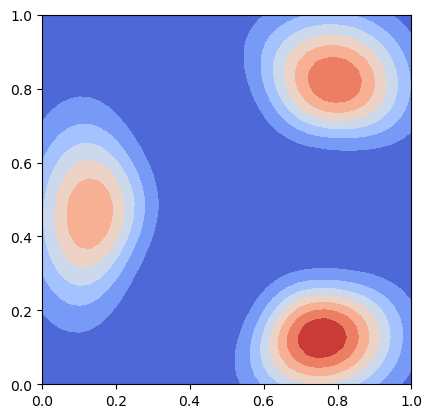}&\includegraphics[width=.249\linewidth]{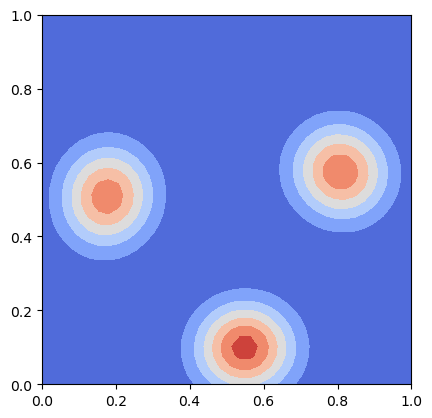}\\
    $t=5/8$&$t=3/4$&$ t=7/8$&$t=1$}

    \caption{Support of $\rho^1$, $\rho^2$ and $\rho^3$ for $V(x,y)=120\chi_{||x-y||<0.2}(x,y)$. }
    \label{fig:support_strongV_3}
\end{figure}

\appendix
\section{Entropy and  the De la Vall\'ee Poussin Theorem}
The existence of minimizers for $J$ relies on the compactness of minimizing sequences that follows from the super-linearity of the entropy functional. 
This is a classical fact which we shortly report here for the reader's convenience. 

Let $\mu$ and $\sigma$ be two probability measures on a metric space $X$. We say that $\mu$ is absolutely continuous with respect to $\sigma$ if there exists a function $f \in L^1_\sigma$ such that
\[
\mu= f (x) \sigma, 
\]
and in this case we write $\mu <<\sigma$ and we use the classical notation $f=\frac{\dd\mu}{\dd\sigma}$. 

The relative entropy of $\mu$ with respect to $\sigma$ is defined as 
\begin{equation*}
\en(\mu | \sigma)= \left\lbrace \begin{array}{ll}
\int_X \frac{\dd\mu}{\dd\sigma} \log \bigg(\frac{\dd\mu}{\dd\sigma}\bigg) \dd \sigma & \mbox{if} \ \mu << \sigma, \\
+ \infty & \mbox{otherwise.}
\end{array}
\right.
\end{equation*}
Following \cite{AmbGigSav} (Example 9.3.6) one can introduce the function
\begin{equation*}
H(s)=\left\lbrace \begin{array}{ll}
s(\log s -1)+1 & \mbox{if} \ s>0,\\ 
1&  \mbox{if} \ s=0,\\
+\infty & \mbox{if} \ s<0,
\end{array}
\right.
\end{equation*}
which is nonnegative, lower semi-continuous, strictly convex and super-linear at $+\infty$. 
Then it holds
\[\en(\mu | \sigma) = \int H\bigg(\frac{\dd\mu}{\dd\sigma}\bigg) \dd \sigma;
\]
and
\[ \en(\mu |\sigma)=0  \leftrightarrow \mu=\sigma.\]
This way to rewrite the relative entropy is handy to make a connection with the De la Vall\'ee Poussin theorem. In fact, 
let $\{\mu_i\}_{i \in I} \subset \Pb (X)$ be such that $\en (\mu_i | \sigma) < C$ then 
the family $\left\{\frac{\dd\mu_i}{\dd\sigma}\right\}_{i \in I} \subset L^1_\sigma (X)$ is weakly compact thanks to the the theorem
\begin{thm}[De la Vall\'ee Poussin]
Let $\{f_i\}_{i \in I} \subset L^1_\sigma$ then the following are equivalent
\begin{itemize}
\item the functions $\{f_i\}_{i \in I} $ are uniformly integrable (and then weakly compact in $L^1_\sigma$ by the Dunford-Pettis theorem,
\item there exists a function $\varphi: \R_+ \to \R_+$ non-decreasing and such that 
\[\lim_{t \to +\infty} \frac{\varphi(t)}{t}=+\infty,\] 
such that  $\int_X \varphi (|f_i|) d\sigma<C$.
\end{itemize}
\end{thm} 
A second advantage of writing the entropy using the function $H$ is the lower semicontinuity with respect to the weak $L^1_\sigma$-convergence 
of the densities, in the sense that if 
\[ f_n \stackrel{w-L^1_\sigma}{\rightharpoonup} f,
\]
then 
\[\liminf_{n \to \infty} \int_X H(f_n) d \sigma \geq \int_X H(f) d \sigma .\]

\section*{Acknowledgment:}
The work of the first author is partially financed by the \textit{``Fondi di ricerca di ateneo''} of the University of Firenze and partially financed by the EU-Next Generation EU, (Missione 4, Componente 2, Investimento 1.1 \textit{Progetti di Ricerca di Rilevante Interesse Nazionale} (PRIN), CUPB53D23009310006 - (2022J4FYNJ). The first author is member of the research group GNAMPA of INdAM.
The second author is partially on academic leave at Inria (team Matherials) for the year 2023-2024 and acknowledges the hospitality of this institution during this period. His work  benefited from the support of the FMJH Program PGMO,  from H-Code, Université Paris-Saclay and from the ANR project GOTA (ANR-23-CE46-0001). 

Both author warmly thanks the hospitality of \textit{equipe MoKaPlan of the INRIA Paris}, where this paper has been written.


\bibliographystyle{plain}
\bibliography{bibli}





  



\end{document}